\documentclass[11pt]{amsart}
\usepackage[pagewise]{lineno}
\usepackage[top = 1in, left = 1.25in, right = 1.5in, bottom = 1in]{geometry}
\usepackage{pdfpages}

\usepackage{amssymb}
\usepackage{amsthm}
\usepackage{amsmath}
\usepackage{etex}
\usepackage{xcolor,colortbl}
\usepackage[all]{xy}
\usepackage{booktabs}
\usepackage{tikz}
\usetikzlibrary{hobby}
\usetikzlibrary{matrix}
\usepackage{array}

\usepackage{amsrefs}
\usepackage{mathrsfs}
\usepackage{verbatim}
\usepackage{graphicx}
	
\def\be{\begin{equation}}
	\def\en{\end{equation}}

\definecolor{darkgreen}{rgb}{.1,.6,0}

\newcommand{\Z}{\mathbb{Z}}

\newtheorem{theorem}{Theorem}[section] % 1st argument is your name for it
\newtheorem{lemma}[theorem]{Lemma}    % 2nd argument is what is printed
\newtheorem{corollary}[theorem]{Corollary}
\newtheorem{proposition}[theorem]{Proposition}

\theoremstyle{definition}
\newtheorem{definition}[theorem]{Definition}
\newtheorem{example}[theorem]{Example}

\newtheorem*{ack*}{Acknowledgment}

\theoremstyle{remark}
\newtheorem{remark}[theorem]{Remark}

\newtheorem{obs}[theorem]{Observation}

\numberwithin{equation}{section}

\title{A classification of nonexpansive Bratteli-Vershik systems}
\author{Karl Petersen}
\address{Department of Mathematics,
	CB 3250 Phillips Hall,
	University of North Carolina,
	Chapel Hill, NC 27599 USA}
\email{petersen@math.unc.edu}
\author{Sandi Shields}
\address{College of Charleston, 66 George St., Charleston, SC 29424-0001 USA}
\email{shieldss@cofc.edu}
\date{\today}


\begin{document}
\begin{abstract}
We study {simple}, properly ordered nonexpansive Bratteli-Vershik ($BV$) systems. 
Correcting a mistake in an earlier paper, we redefine the {classes} standard nonexpansive ($SNE$) and strong standard nonexpansive ($SSNE$).
{We define also the {classes of very well timed and well timed systems}, their opposing classes of untimed and very untimed systems (which feature, as subclasses of ``Case (2)", in the work of Downarowicz and Maass as well as Hoynes on expansiveness of $BV$ systems of finite topological rank), 
and several related classes according to the existence of indistinguishable pairs (of some ``depth") and their synchronization (``common cuts").  
We establish some properties of these types of systems and some relations among them. 
We provide several relevant examples, including a problematic one that is conjugate to a well timed system while also (vacuously) in the classes ``Case (2)".}
We prove that the class of all simple, properly ordered nonexpansive $BV$ systems is the disjoint union of the ones conjugate to well timed systems and those conjugate to untimed systems, 
thereby showing that 
nonexpansiveness in $BV$ systems arises in one of two mutually exclusive ways. 

\end{abstract}
\maketitle

	\section{Introduction} 
Bratteli-Vershik ($BV$) systems present visually and combinatorially the hierarchical mechanisms that drive measure-preserving and topological dynamical systems. 
Vershik \cites{Vershik1981Uniform,Vershik1981Markov} and Herman-Putnam-Skau \cite{HPS1992} showed that every measure-preserving system and every minimal homeomorphism on the Cantor set is isomorphic (in the appropriate sense) to a $BV$ system.
As with any system, it is useful, when possible, to code the dynamics as a subshift, so that the methods of symbolic dynamics and formal languages can be brought into action. 
This is possible exactly when the system is expansive (in the topological setting), or essentially expansive (in the measure-preserving setting; see \cite{AFP}). 
There has been considerable progress on the question of when $BV$ systems are expansive or not, or when a sequence of morphisms is recognizable: \cites{dm2008,BezuglyiKwiatkowskiMedynets2009,FPS2017,AFP,Berthe2017,FPS2020}, for example. 
Here we explore the reasons why a $BV$ system might be nonexpansive.
We focus on $BV$ systems that are simple (so the topological dynamical system is minimal) and properly ordered (there is a unique minimal path and a unique maximal path), but not necessarily of finite topological rank (conjugate to one with a bounded number of vertices per level).

Many nonexpansive systems are conjugate to odometers, which {can be represented by diagrams of}  bounded width, or seem to be similar to the Gjerde-Johansen example \cite{GJ2000}*{Figure 4}, which has unbounded width 
{and is not conjugate to an odometer}.
For brevity we will refer to this system, which motivated much of this investigation, as the {\em GJ example}.
In an earlier paper \cite{FPS2017}, abstracting key properties of this example, {the authors (including us)} proposed a class of systems (standard nonexpansive, $SNE$) as a model for how nonexpansivesness can arise in $BV$ systems.
{Proposition 5.4 of \cite{FPS2017} asserted that every standard nonexpansive system, according to the definition given there, has unbounded width and cannot be conjugate to an odometer.
	Example \ref{ex:OldSNE} below shows that neither of these properties is guaranteed under the old definition of $SNE$;  
	that definition was flawed, {being based on the assumption that} two paths with the same sequence of ordinal edge labels had to move simultaneously to new vertices.
	We cannot even guarantee that $SNE$ systems, as originally defined, are nonexpansive.
	Here we correct those mistakes by providing a new definition of $SNE$ (Definition \ref{def:newSNE}), as well as its invariant version $SSNE$ {(Definition \ref{def:SSNE}).}
	Nonexpansiveness now follows from Proposition \ref{prop:samecodings}, and  Proposition  \ref{prop:noSNEconjOdom} and Theorem \ref{thm:unbddwidth} show that any system satisfying the new definition of $SNE$ is not conjugate to any odometer and has unbounded width.}
\begin{comment}
{It was claimed there} that all such systems had to have unbounded width, but Example \ref{ex:OldSNE} below shows that this is not so;
{the definition there was flawed, assuming that two paths with the same sequence of ordinal edge labels had to move simultaneously to new vertices.}
Here we correct those mistakes by providing a new definition of $SNE$, as well as its invariant version $SSNE$, and a new theorem (Theorem \ref{thm:unbddwidth}) replacing the previous mistaken one. 
{In Observations \ref{obs:GJ}, \ref{obs:aperiodic}, and \ref{obs:recog} and Proposition \ref{prop:GJSNE} we detail properties of the GJ example that allow us to present in Example \ref{ex:modGJ} a modification that is not SNE and is not conjugate to any odometer.}
\end{comment}

{In Section \ref{sec:modGJ} we study the GJ example in detail, to prove that it is standard nonexpansive and to show how a slight modification ruins this property. 
	For the latter, we use the technique of splitting $j$-symbols from \cite{dm2008} and prove that the $2$-coding of every path in the GJ example is aperiodic (Observation \ref{obs:aperiodic}) 
	and the related fact that a certain factor system gives rise to a recognizable family of morphisms (Observation \ref{obs:recog}).}

{We define two types of systems, well timed or untimed, as well as their subclasses very well timed and very untimed. 
{The GJ example is very well timed.}
The untimed systems are related to what Downarowicz and Maass called ``Case (2)" in \cite{dm2008}. 
Theorem \ref{thm:disjunion} shows that the family of all nonexpansive systems is the disjoint union of the class of systems which are conjugate to some well timed system and the class of those conjugate to some untimed system.} 

 {Hoynes \cite{Hoynes2017} used a slightly different ``Case (2)", and both \cite{dm2008} and \cite{Hoynes2017} used telescoping to arrive at stronger properties. 
  Definition \ref{def:classes} states the definitions of these properties and a few of the others that would arise from permuting or negating the quantifiers concerning depth and cuts. 
 Proposition \ref{prop:newlemma} clarifies what happens to depth and cuts under telescoping. 
 Working towards the proof of Theorem \ref{thm:disjunion}, we characterize in Proposition \ref{prop:CW} the systems that are conjugate to some well timed system as those that are ``weakly well timed".  
 We establish relations among these various types of systems. Example \ref{ex:odoms} shows that odometers can be either {untimed but not very untimed, or very untimed}, depending on how they are presented. 
 Example \ref{ex:DM2WW} is a system that is both weakly well timed and (vacuously) satisfies the property in each ``Case (2)"; this indicates that, as written, the proofs of the main theorem of \cites{dm2008,Hoynes2017} are slightly incomplete.}

{In Section \ref{sec:untimed} we describe the possible diagrams for bounded width very untimed systems: they are all ``kite shaped", like the example in Figure \ref{fig:nondet}.
	In the final section we mention several questions suggested by the foregoing, for example  whether there are any {untimed} systems that have unbounded width and whether every well timed system is conjugate to a system {that is in $SNE$.}
	
	\begin{ack*}
		We thank Sarah Bailey Frick for valuable contributions during the first half of this project.
		\end{ack*}

\section{Some definitions and notation}\label{sec:Defs}

{We deal with Bratteli-Vershik ($BV$) systems $(X,T)$.
{Each system is built on an ordered Bratteli diagram, which is a countably infinite, directed, graded, graph. 
For each $n=0,1,2,\dots$ there is a finite nonempty set of vertices $V_n$.
$V_0$ consists of a single vertex, called the ``root". 
The set of edges is the disjoint union of finite nonempty sets $E_n, n\geq 0$, with $E_n$ denoting the set of edges with source in $V_n$ and target in $V_{n+1}$. 
We assume that every vertex has at least one outgoing edge, and every vertex other than the root has at least one incoming edge. 
There can be multiple edges between pairs of vertices.} 
	{The space $X$ is the set of infinite paths (sequences $x=x_0x_1\dots$, each $x_i$ denoting an edge from level $i$ to level $i+1$) starting at the root {at level $i=0$}}. 
	For a path $x$ we denote by $v_i(x)$ the vertex of the path at level $i$.
	{If $i<j$, we will say that level $j$ is {\em after} or {\em later than} level $i$, which is {\em earlier} or {\em before} level $j$. 
		Since diagrams are often drawn with later levels below earlier ones, we may sometimes say that level $j$ is {\em below} level $i$ if $i<j$ and use {\em up} and {\em down} to refer to relative positions in such diagrams.}
	$X$ is a compact metric space when we specify that two paths have distance $1/2^n$ if they agree from levels $0$ to $n$ and disagree {leaving level} $n$. To avoid degenerate situations we assume that $X$ is homeomorphic to the Cantor set.}
	
	{The edges entering each vertex are totally ordered, and this yields a partial order on the set of infinite paths as follows. 
	Two paths $x$ and $y$ are {\em comparable} in case they are cofinal: there is a smallest $N>0$ such that $x_n=y_n$ for all $n \geq N$. 
In this case $v_{N}(x)=v_{N}(y)$, and $x_{N-1} \neq y_{N-1}$; we agree that $x<y$ if $x_{N-1} < y_{N-1}$, and $x>y$ if not.
The set of {\em minimal paths}, meaning those all of whose edges are minimal into all of their vertices, will be denoted by $X_{\min}$, and similarly the set of {\em maximal paths} will be denoted by $X_{\max}$. 
The Vershik map $T$ is defined from the set of nonmaximal paths to the set of nonminimal paths by mapping each path $x$ to its {\em successor}, the smallest $y>x$.
We assume that the diagrams are {\em properly ordered}, which means that 
there is a unique minimal path and a unique maximal path.
This implies that the Vershik map is {\em perfectly ordered}: it extends to a homeomorphism $T$ on the set $X$ of all infinite paths from the root by mapping the unique maximal path to the unique minimal path. (See \cites{BezuglyiKwiatkowskiYassawi2014,BezuglyiYassawi2016}).} 

The system is called {\em minimal} if every orbit is dense, and it is called {\em simple} if it has a telescoping for which there are complete connections between adjacent levels. 
{It can be proved that a} properly ordered system is minimal if and only if it is simple. 

We will say that two systems are {\em conjugate} if there is an equivariant homeomorphism from one to the other that sends minimal points to minimal points. 
The relation of {\em diagram equivalence} is the one generated by graph isomorphism and telescoping. 
 Herman-Putnam-Skau \cite{HPS1992}*{Section 4} (see also \cite{GPS1995}*{p. 70 and p. 72, Theorem 3.6}) 
 characterized conjugacy of minimal pointed topological dynamical systems defined by properly ordered (called there ``essentially simple") $BV$ diagrams {(with unique minimal points)}, as follows:
 {two systems are conjugate if and only if their diagrams are equivalent, and this happens if and only if there exists a diagram $Z$ that telescopes on odd levels to a telescoping of one of the diagrams and on even levels to a telescoping of the other.} 
 
{We emphasize that in the following, unless stated otherwise, every {\em system} is a $BV$ system defined by a properly ordered, simple diagram, and every {\em conjugacy} is an equivariant homeomorphism between two such systems that sends {the unique minimal point in one system to the unique minimal point in the other} (but we may include reminders about these hypotheses anyway).}
{For convenience we may use the same symbol, such as $X$ or $Y$, to denote a diagram as well as the system that it defines, and we may also use the same symbol, such as $T$, for Vershik maps on different systems, or suppress it entirely.
{By ``pair" we mean a pair of distinct elements, unless stated otherwise.
The {\em width} of a level is the number of vertices at that level.  
Recall that in \cite{dm2008} the {\em topological rank} of a system $(X,T)$ is defined to be the minimum among all $BV$ systems $(Y,T)$ conjugate to $(X,T)$ of the supremum of the widths of the levels of $(Y,T)$.} 
For further terminology and background, see \cites{HPS1992,GPS1995,Durand2010,FPS2017} and the references cited there.}

 {For each $k \geq 1$ denote by $A_k$ the finite alphabet whose elements are the finite paths (segments, strings of edges) from the root to level $k$. 
	For each $a \in A_k$, the set {$[a]=\{x \in X: x_0 \dots x_{k-1}=a\}$} is a clopen {\em cylinder set}, and $\mathcal P_k=\{ [a]:a \in A_k\}$ is a partition of $X$ into clopen sets.  
	The map $\pi _k:X \to A_k^\Z$ is defined by $(\pi_kx)_n=a$ if and only if $T^nx \in [a]$.
	{The doubly infinite sequence $\pi_kx$ is called the {\em $k$-coding of $x$.}}
	Let $\Sigma_k=\pi_kX$, and denote by $\sigma$ the shift transformation on $A_k^\Z$. 
	Then $\pi_k:(X,T) \to (\Sigma_k, \sigma)$ is a {\em factor map}: it is continuous, onto, and it commutes with the transformations.} 

{\begin{definition}\label{def:bb}
For each vertex $v$ {at level $k \geq 1$}, denote by $P(v)$ the set of paths from the root to $v$, {define the {\em dimension} of the vertex $v$ to be }$\dim v=|P(v)|$, and concatenate the elements of $P(v)$, in their {lexicographic order} (defined by the edge ordering), as $p_1 \dots p_{\dim v}$ (so that $[p_{j+1}]=T[p_j]$ for $j=1, \dots ,\dim v -1$). We call the string $p_1 \dots p_{\dim v}$ on symbols from the alphabet $A_k$ the {\em $k$-basic block} at $v$ {and denote it by $B_k(v)$}.
{Each path from the root to level $k$ determines, by truncation, for each $i<k$ a unique path from the root to level $i$, so there is {a} natural factor map $A_k \to A_i$ which converts $B_k(v)$ to a string, {which we denote by $B_i(v)$ and call the {\em $i$-basic block at $v$}}, of the same length on the alphabet $A_i$.}
\end{definition}}

\begin{definition} \label{def:codingbyvertices}
	{The {\em coding by vertices at level $j<n$ of a vertex $w$ at level $n$}, denoted by $C_j(w)$, is defined as follows. 
List in their order in the diagram the paths entering $w$ from vertices at level $j$ as $\{p_1,\dots ,p_m\}$ and denote the source of $p_i$ by $v_i, i=1,\dots ,m$.
Then $C_j(w)=v_1 \dots v_m$. 
(For $j=n-1$, this is the ``morphism read on $\mathcal V_{n}$" in \cite{Durand2010}*{p. 342}).}
		\end{definition}

\begin{definition}\label{def:NE}
	 We say that a  $BV$ system is {\em nonexpansive} (abbreviated $NE$) if for every $k \geq 1$ there exists a pair of paths with the same $k$-coding (by finite paths from level $0$ to level $k$).  \end{definition}

 Thus a system is expansive if and only if there is a $k \geq 1$ such that the map $\pi_k: (X,T) \to (\Sigma_k,\sigma)$ is injective (and hence a conjugacy). 

	\begin{definition} We say two paths $x$ and $x'$ are {\em depth $k \geq 0$} if they have the same $k$-coding but not the same $(k+1)$-coding. \end{definition}
	
	If two paths have the same $k$-coding, then they agree from the root to level $k$.
	
	For any pair of (distinct) paths $x$ and $x'$  there exists a $j$ such that $x$ and $x'$ do not have the same $j$-coding. Hence, telescoping can be used to convert a nonexpansive $BV$ system into one with the property that for every $k\geq 1$ there exists a depth $k$ pair of paths.

	\begin{definition}
		We say two paths $x$ and $x'$ have a {\em (common) $j$ cut} if there is an integer $m$ such that the initial segments of $T^mx$ and  $T^mx'$ are minimal into level $j$.   
		Then we say that the pair has a cut {\em {at time} $m$}.
		\end{definition}

{We say that two paths {\em differ at level $j$} if they follow different edges into level $j$.}
Note that if two paths $x$ and $x'$ {\em differ} at level $k+1$ and have a $k+1$ cut, then $m$ can be chosen so that $T^mx$ and  $T^mx'$ are minimal into {\em distinct} vertices at level $k+1$. (For detailed explanation, see top of page 7 in \cite{Hoynes2017}). 

\section{Standard nonexpansive systems}\label{sec:SNEsystems}
\begin{definition}\label{def:suo}
	Let $k\geq 1$ and $n>k$. We say that two vertices $v,w$ at level $n$ are {\em $k$-equivalent}, 
	and write $v\sim_k w$,
	 if they are {\em strongly uniformly ordered with respect to level $k$}, meaning that the set of paths from vertices at level $k$ to $v$ is order isomorphic with the set of paths from vertices at level $k$ to $w$. Equivalently, the $k$-basic blocks at $v$ and $w$ are identical.  
   (In \cite{FPS2017} a {\em uniformly ordered} level $n+1$ was defined to be {a level} all of whose vertices had $n$-basic blocks that were powers of a single block on the alphabet of paths from the root to level $n$.) 
\end{definition}
	
	 Any pair of vertices that is $(k+1)$-equivalent is also $k$-equivalent. (If $v$ and $w$ at level $n>k+1$ have equal basic blocks in terms of the paths from the root to level $k+1$, when these blocks are rewritten in terms of paths to level $k$ the results will still be identical.)
	 
	 \begin{definition}
	{We say that two paths $x$ and $x'$ are {\em $k$-equivalent at level $n$} if we have $v_n(x) \sim_k v_n(x')$ 
	  {\em and moreover} their paths from the root to level $n$ have {\em the same ordinal path label}. 
		 We say  $x$ and $x'$ are {\em $k$-equivalent}, and write $x \sim_k x'$, if they agree from the root to level $k$ and, for all $n>k$, $x$ and $x'$ are $k$-equivalent at level $n$.}
	 \end{definition}
 
\begin{remark}\label{remark:dot} 
		{Recall that for a path $x$ and $n \geq k \geq 1$, the $k$-basic block $B_k(v_n(x))$ lists in their assigned order the truncations to levels from $0$ to $k$ of paths from the root to $v_n(x)$, say as $q_1 \dots q_{\dim (v_n(x))}$ (see Definition \ref{def:bb}).
		Suppose that $i$ is the index in $[1,\dim (v_n(x))]$ for which $q_i$ is the initial segment of $x$ from the root to level $n$. 
		This information can be conveyed by writing out the string $B_k(v_n(x))$ with {a} ``dot" immediately preceding $q_i$. For example, if $x$ follows only minimal edges to level $n$, then $B_k(v_n(x))=.q_1 \dots q_{\dim(v_n(x))}$.}
	{Two paths $x,x'$ are $k$-equivalent at level $n>k$ if $B_k(v_n(x))=B_k(v_n(x'))$ and, informally, these two basic blocks have the ``dot" in the same place}.
 \end{remark}	
 
	 \begin{definition}\label{def:newSNE}
	  	We say that a nonexpansive $BV$ system is {\em standard nonexpansive (SNE)} if for every $k \geq 1$ there is a pair of $k$-equivalent paths.
\end{definition}

\begin{definition}\label{def:SSNE}
	We say that two paths $x$ and $x'$ are {\em $k$-same}, and write $x \approx_k  x'$, if $T^j x \sim_k T^j x'$ for all $j \in \mathbb Z$. 
	We say that a system is {\em  strong standard nonexpansive (SSNE)} if for every $k \geq 1$ there is a pair of distinct $k$-same paths.
\end{definition}

\begin{remark}
	The paper \cite{FPS2017} introduced a {different} definition of standard nonexpansive: For every $k$ there should exist a pair of paths that agree from the root to {a level $n>k$}, have the same sequence of edge labels, and have the same $k$-basic blocks at all levels 
	{after $k$}.  
	{In the new definition (\ref{def:newSNE}) we replace the requirement that the sequences of edge labels be the same with the stronger requirement that
		 the two paths always have the ``dot" in the same place of their equal $k$-basic blocks.}\\
	\indent	
	Proposition 5.4 of \cite{FPS2017} asserted that every standard nonexpansive system, according to the definition given there, has unbounded width. 
	The following example shows that this is not correct. 
	In Theorem \ref{thm:unbddwidth} below we show that the {new}  definition of $SNE$ given here does suffice to guarantee unbounded width.
	\begin{comment}
{Proposition 5.4 of \cite{FPS2017} asserted that every standard nonexpansive system, according to the definition given there, has unbounded width and cannot be conjugate to an odometer.
	 Example \ref{ex:OldSNE} below shows that neither of these properties is guaranteed under the old definition of $SNE$.  
	 We cannot even guarantee that $SNE$ systems, as originally defined, are nonexpansive.
	 Nonexpansiveness now follows from Proposition \ref{prop:samecodings}, and  Proposition  \ref{prop:noSNEconjOdom} and Theorem \ref{thm:unbddwidth} show that any system satisfying the new definition of $SNE$ is not conjugate to any odometer and has unbounded width.}
	\end{comment}

	\begin{example}\label{ex:OldSNE}
	{This example satisfies the old definition of $SNE$, but not the new one. 
		It has bounded width and is conjugate to an odometer {(since after telescoping to the levels with exactly two vertices, all levels are uniformly ordered, see \cite{FPS2017})}; so it is a counterexample to Prop 5.4 of \cite{FPS2017} and the assertion in its  proof that any system conjugate to an odometer cannot have pairs of paths with the same edge labels.}\\
	\indent
{In this example (see Figure \ref{fig:ig}) the dark path and the dotted path are identical from the root to the vertex labeled {$b$ at level $3$}, where they diverge. 
They continue below with the same sequence of ordinal edge labels. 
At level $1$ (not shown) there are two vertices, $u$ and $v$, 
		and at level $2$ there are four vertices, with {codings by vertices of the previous level} $u^2, u^3, uv^2,v^2u$.
		The vertices $u$ and $v$ connect to $a$ and $b$ via edges in the order that leads to identical {codings by vertices}  $u^2uvvu^3$ at $a$ and $u^3vvuu^2$ at $b$.
At all levels $2$ and later, the paths have identical {codings by vertices on the symbols $u$ and $v$}, but the ``dots" for the two paths are in different places after level $4$. 
After level $0$, the diagram repeats levels and edges with period two, so it has bounded width.\\
\indent
To find such pairs of paths for values of $k$ larger than $1$, use the periodicity of the diagram. For example, to deal with $1$ replaced by $3$ (and hence also by $2$), we may form a pair of paths that is identical from the root to the vertex labeled $B$ at level $5$, where it splits into two paths following the same sequence of edge labels as before ($22121212\dots$). 
These paths always enter vertices with the same {codings by level $3$ vertices} $a,b$, but not with the dot in the same place {(after level $6$)}. 
Then for $3$ replaced by $5$ (and hence also $4$), take two paths that follow identical edges from the root to the vertex $D$ at level $7$, where they split and follow the same sequence of edges $22121\dots$. 
These paths always enter vertices with the same {codings} by {level $5$ vertices} $A,B$, and hence also by the vertices at level $4$, but not with the dots in the same place. 
Etc. }
	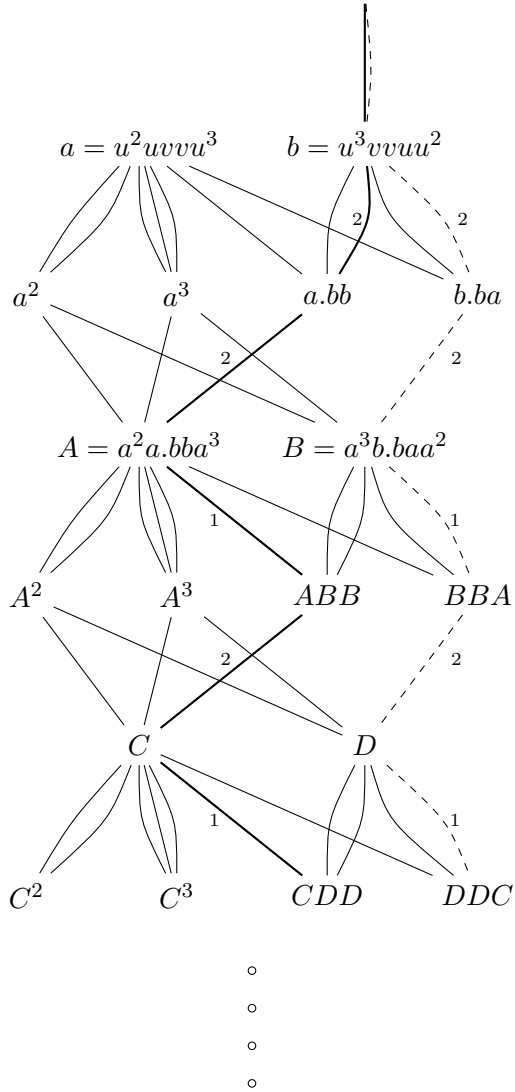
\begin{figure}
	\begin{tikzpicture}[scale=1]
	\node at (4.5,12) (alpha) {};
	\node at (1.5,10) (a) {$a=u^2uvvu^3$}; \node at (4.5,10) (b) {$b=u^3vvuu^2$};
	\node at (0,8) (a^2) {$a^2$}; \node at (2,8) (a^3) {$a^3$}; \node at (4,8) (abb) {$a.bb$}; \node at (6,8) (bba) {$b.ba$};
	\node at (1.5,6) (A) {$A=a^2a.bba^3$}; \node at (4.5,6) (B) {$B=a^3b.baa^2$};
	\node at (0,4) (A^2) {$A^2$}; \node at (2,4) (A^3) {$A^3$}; \node at (4,4) (ABB) {$ABB$}; \node at (6,4) (BBA) {$BBA$};
	\node at (1.5,2) (C) {$C$}; \node at (4.5,2) (D) {$D$};
	\node at (0,0) (C^2) {$C^2$}; \node at (2,0) (C^3) {$C^3$}; \node at (4,0) (CDD) {$CDD$}; \node at (6,0) (DDC) {$DDC$};
	
		\draw [thick] (alpha) -- (b); \draw [dashed] (alpha) .. controls (4.6,11) .. (b);
	%\draw [thick] (alpha) .. controls (4.2,11) .. (b); \draw [dashed] (alpha) .. controls (4.7,11) .. (b);
	\draw [thick] (b) .. controls (4.6,9) .. (abb); \draw [dashed] (b) .. controls (5.6,9) .. (bba);
	\draw [thick] (abb) -- (A); \draw [dashed] (bba)  -- (B);
	\draw [thick] (A) -- (ABB); \draw [dashed] (B) .. controls (5.6,5) .. (BBA);
	\draw [thick] (ABB) -- (C); \draw [dashed] (BBA) -- (D);
	\draw [thick](C) -- (CDD); \draw [dashed] (D) .. controls (5.6,1) .. (DDC);
	
	\draw (a) .. controls (1.1,9) .. (a^2); \draw (a) .. controls (.6,9) .. (a^2);
	\draw (a) -- (a^3); \draw (a) .. controls (1.5,9) .. (a^3);\draw (a) .. controls (2,9) .. (a^3);\draw (a) -- (abb);
	\draw (b) .. controls (4,9) .. (abb); \draw (b) .. controls (4.8,9) .. (bba);
	\draw (a) -- (bba);
	
	\draw (a^2) -- (A); \draw (a^3) -- (A); \draw (a^3) -- (B); \draw (abb) -- (A); 
	\draw (a^2) -- (B);
	
	\draw (A) .. controls (1.1,5) .. (A^2); \draw (A) .. controls (.6,5) .. (A^2);
	\draw (A) -- (A^3); \draw (A) .. controls (1.5,5) .. (A^3);\draw (A) .. controls (2,5) .. (A^3);\draw (A) -- (ABB);
	\draw (B) .. controls (4,5) .. (ABB); \draw (B) .. controls (4.5,5) .. (ABB); \draw (B) .. controls (4.8,5) .. (BBA);
	\draw (A) -- (BBA);
	
	\draw (A^2) -- (C); \draw (A^3) -- (C); \draw (A^3) -- (D); \draw (ABB) -- (C);
	\draw (A^2) -- (D);
	
	\draw (C) .. controls (1.1,1) .. (C^2); \draw (C) .. controls (.6,1) .. (C^2);
	\draw (C) -- (C^3); \draw (C) .. controls (1.5,1) .. (C^3);\draw (C) .. controls (2,1) .. (C^3);\draw (C) -- (CDD);
	\draw (D) .. controls (4,1) .. (CDD); \draw (D) .. controls (4.5,1) .. (CDD);\draw (D) .. controls (4.8,1) .. (DDC);
	\draw (C) -- (DDC);
	
	\node at (3,-1) (x) {};
	\draw  (x) circle  (0.05);
	\node at (3,-1.5) (s) {};
	\draw  (s) circle (0.05);
	\node at (3,-2) (y) {};
	\draw  (y) circle (0.05);
	\node at (3,-2.5) (z) {};
	\draw  (z) circle (0.05);
	%\node at (0,-6.6) (r) {};
	%\draw  (r) circle (0.05);
	
	\node at (4.4,9) {{\tiny $2$}};\node at (5.8,9) {{\tiny $2$}};
	\node at (2.65,7.15) {{\tiny $2$}};\node at (5.72,7.15) {{\tiny$2$}};
	\node at (2.5,5) {{\tiny $1$}};\node at (5.7,5) {{\tiny $1$}};
	\node at (2.65,3.15) {{\tiny $2$}};\node at (5.72,3.15) {{\tiny$2$}};
	\node at (2.5,1) {{\tiny $1$}};\node at (5.7,1) {{\tiny $1$}};
		\end{tikzpicture}
	\caption{Part of a diagram whose system satisfies the old definition of $SNE$ but not the new one}
	\label{fig:ig}
	\end{figure}
			\end{example}
	\end{remark}
	
	{The strict synchronizing structure of $k$-equivalent {pairs of paths} guarantees that they cannot be separated by their $k$-codings.}
	\begin{proposition}\label{prop:samecodings}
		{In any system, if two paths are 
			$k$-equivalent at infinitely many levels, then they have the same $k$-codings.} 
	\end{proposition}
	\begin{proof}
		{Suppose that $x,x'$ is a pair of {distinct} paths for which there is an infinite increasing sequence $(n_j)$ such that 
			for each $j$ the paths $x,x'$ are $k$-equivalent at level $n_j$. 
			Their (identical) $k$-basic blocks at their $n_j$-level vertices $v_{n_j}(x),v_{n_j}(x')$ have lengths increasing to infinity, $j \geq 1$. 
			If the lengths of the segments of the blocks both to the left and right of their dots increase unboundedly, $x$ and $x'$ will have identical $k$-codings. 
			Suppose that the segments to one side of the dot, for example the left, stay bounded. 
			Then there are $m \in [0,\infty)$ and $j_0$ such that the segment to the left has length $m$ at all levels $n_j$ for $j \geq j_0$.
			At these levels $T^{-m}x,T^{-m}x'$ are minimal paths from the root to the vertices $v_{n_j}(x),v_{n_j}(x')$. 
			Since $n_j$ is arbitrarily large, $T^{-m}x$ and $T^{-m}x'$ consist entirely of minimal edges, so both equal the minimal path $x_{\min}$,
			{contradicting our assumption that the paths are distinct.
				Analogously, using uniqueness of the maximal path, 
			the segments of the basic blocks to the right cannot stay bounded.}}
	\end{proof} 

\begin{corollary}\label{cor:sneimpliesne}
	{Standard nonexpansive implies nonexpansive.}
	\end{corollary}

	\begin{proposition}\label{prop:eqcut}
		If two paths $x$ and $x'$ are $k$-equivalent, then 
		for every $n>k$ they have an $n$ cut, and there is
		a $j \geq k$ such that they are depth $j$. 
		Hence, in any system, {if two $k$-equivalent paths are depth $n \geq k$,} then they {have an $n+1$ cut.}
	\end{proposition}
	\begin{proof}  			
		For any path and any $n>k$, the position of the dot in the $k$-basic block at level $n$ represents the
		smallest number of applications of $T^{-1}$ applied to that path necessary to produce a path that is minimal into level $n$. Since $x$ and $x'$ are $k$-equivalent, the dot is in the same position for their $k$-basic blocks at  level $n$, so the paths  $x$ and $x'$ have an $n$ cut.  
		
		Given paths $x$ and $x'$ in a system that are $k$-equivalent, {by Proposition \ref{prop:samecodings} they have the same $k$-codings.} There exists  a largest $j \geq k$ such that for all $m \in \Z$, $T^mx$ and $T^mx'$ agree to level $j$. It follows that $x$ and $x'$ are depth $j$. 
	\end{proof}
	
	{Note that the converse {of each statement in Proposition \ref{prop:eqcut}} is not true: it can happen {for a depth $k$ pair and some $n>k$} that the $k$-basic block at $v_{n}(x)$ is a proper prefix of the $k$-basic block at $v_{n}(x')$, and then the paths would not be $k$-equivalent {but could still have an $n$ cut.}

		The following Proposition will be extended in Theorem \ref{thm:CWimpliesNCU}.
		
		\begin{proposition}\label{prop:noSNEconjOdom}
			No $SNE$ system can be conjugate to any odometer.
		\end{proposition}
		\begin{proof}
			{Suppose that $(X,T)$ is a system that is conjugate to an odometer. By \cite{FPS2017}*{Theorem 5.3} it has a telescoping {that has infinitely many uniformly ordered levels}. 
				We claim that for $k = 1$ no pair in the telescoped system can be $k$-equivalent, so the 
				{telescoped system} in fact cannot be $SNE$.}
			
			{Suppose that $x,x'$ is a $1$-equivalent pair {in the telescoped system}. Since $x \neq x'$ 
				{but their first edges are equal, $x_0=x'_0$}, the pair first disagrees at some level $j > 1$, that is to say, they follow different edges from level $j-1$ to level $j$.
				If they enter the same vertex at level $j$ along these different edges, they cannot have their dots in the same place. So the paths are at different vertices at level $j$, and then, by induction, they must be at different vertices at each level $n$ for all $n \geq j$.}
			
			Look at the first uniformly ordered level $n > j$. 
			At the vertices $v_{n}(x)$ and $v_{n}(x')$, {because of $1$-equivalence we have $(n-1)$-basic blocks of the same length, and then because of uniform order these $(n-1)$-basic blocks are equal.} 
				Since $x,x'$ are at different vertices at level $n-1$, their dots are at different places in the {$(n-1)$}-basic blocks at $v_{n}(x), v_{n}(x')$. 
				When {these} {$(n-1)$}-basic blocks are expanded to $1$-basic blocks, the dots for $x,x'$ will be at different places, {contradicting $1$-equivalence of $x$ and $x'$.}	
		\end{proof}
		
		{The following Theorem follows directly from the preceding Proposition combined with the main result of \cite{dm2008}, but we give a direct proof here in order to re-establish Proposition 5.4 of \cite{FPS2017} in our context, with the new definition of $SNE$.} 
		
		\begin{theorem}\label{thm:unbddwidth}
			SNE implies unbounded width.
		\end{theorem}
		\begin{proof}
			{Starting with $j_0=1$, {using Proposition \ref{prop:eqcut}}, we can find two paths $x^{(0)},y^{(0)}$ that are $j_0$-equivalent and follow different edges into some level $j_1 > j_0$.
				As above, because they are $j_0$-equivalent, they must pass through different vertices at level $j_1$, and hence they pass through different $j_0$-equivalent vertices for all $j \geq j_1$.}

			{If they pass through $j_1$-equivalent vertices at infinitely many levels, 
				they must have the dot in the same place in their basic blocks at those levels, because when the $j_1$-basic blocks are expanded to $j_0$-basic blocks, they are the same and have the dot in the same place, so the dot had to be in the same place for the $j_1$-basic blocks.
				Then $x^{(0)}$ and $y^{(0)}$ would be $j_1$-equivalent at these infinitely many levels, and so by Proposition \ref{prop:samecodings} would have the same $j_1$-coding, but they do not. 
				Thus for all large enough $n$, $x^{(0)}$ and $y^{(0)}$ pass through vertices that are $j_0$-equivalent and not $j_1$-equivalent. }
						
			Continuing, given $K \geq 1$, for each $k=0,\dots,K-1$ we can find $N$, integers $j_0 < j_1 < \dots < j_{K}$, and pairs $x^{(k)},y^{(k)}$ that for each $n \geq N$ and $k=0, \dots, K-1$ pass through vertices that are $j_k$-equivalent and not $j_{k+1}$-equivalent:
				\be
				v_n(x^{(k)})\sim_{j_k} v_n(y^{(k)}) \quad\text{ while }\quad v_n(x^{(k)})\nsim_{j_{k+1}} v_n(y^{(k)}).
				\en
						{But these pairs of vertices at any level $n \geq N$ are all different, because not $j_k$-equivalent implies not $j_{k+1}$-equivalent. If there are at most $R$ vertices at every level and $K> R^2$, this is a contradiction.}
		\end{proof}

\section{Modifing the GJ example to ruin the property $SNE$}\label{sec:modGJ}
		{We will show that the
example of Gjerde and Johansen \cite{GJ2000}*{Figure 4}, which has unbounded width and is not conjugate to any odometer, is $SNE$ (even $SSNE$) by specifying for any $k$ exactly which pairs of paths are $k$-equivalent. We will also determine for any $k$ which pairs are depth $k$.}

{Denote the system in the GJ example  by $X$.
Given $j>0$ there are $2j$ vertices in the diagram at level $j$. 
For any $i \leq 2j$, let $v(j,i)$ denote the $i$'th vertex from the left (beginning with $i=1$) in level $j$.
For $i=1,2,\dots$ and $j>i$, paths that at level $j$ pass only through vertices $v(j,2i)$ or $v(j,2i+1)$ will be said to constitute the {\em $i$'th Morse component $MC(i)$} of the diagram. In particular, $MC(i)$ begins at level $i+1$.}

\begin{obs}\label{obs:GJ}
{(1) For any $i \geq 1$, at any level of $MC(i)$ the $i$-basic blocks are the same (at both vertices), as can be verified by writing them out. It can also be verified that the $(i +1)$-basic blocks at the vertices of $MC(i)$ differ after level $i+1$.}

{(2) In the diagram for $X$, each vertex is connected to every vertex at the previous level by exactly one edge, {so all vertices at any level have the same dimension}. From this it can be argued inductively that for every $j \geq 1$ two paths have the same ordinal path label {from the root} into level $j$ precisely when they have the same sequence of edge labels into level $j$.}
\end{obs}

{\begin{proposition}\label{prop:GJSNE}.
		The following statements hold for the GJ example, $X$.\\
		{(1) If two edges at level $n \geq 2$ have the same label, {then they have different targets. If they also have different sources, then they are contained in $MC(j)$ for some $j<n.$}}\\
		(2) A pair of paths $x$, $x'$ with the same sequence of edge labels first differ at {level $j \geq 1$} if and only if they enter a Morse component at level $j$ {and} remain in that Morse component at all levels $j$ and higher,
		{never again meeting the same vertex}.\\
		{(3) If two paths have the same sequence of edge labels {and first differ at level $j$, then for some $k<j$ they enter $MC(k)$ at level $j$} and are $k$-equivalent and not $(k+1)$-equivalent.
			Conversely, if two paths are $k$-equivalent and not $(k+1)$-equivalent, then they have the same sequence of edge labels {and enter $MC(k)$ at the level where they first differ}.} \\
			(4) For any $k \geq 2$, a pair of paths is depth $k$ if and only if those paths are $k$-equivalent but not $(k+1)$-equivalent.  
	{(Equivalently, in view of {(3)},} two paths are depth $k$ if and only they have the same sequence of edge labels and enter $MC(k)$ at the level where they first differ.)
	\end{proposition}
\begin{proof}
{Proof of (1):
Fix $n \geq 2$. All edges with source $v(n,1)$ are labeled 1 and all edges with source $v(n, 2n)$ are labeled $2n$. 
For all $j = 1, ...,n-1$, all edges with source $v(n, 2j)$ are labeled $2j$, with the exception of the single edge between $v(n, 2j)$ and $v(n+1, 2j+1)$, which is labeled $2j+1$.  
Likewise, all edges with source $v(n, 2j+1)$ are labeled $2j+1$, with the exception of the single edge between $v(n, 2j+1)$ and $v(n+1, 2j+1)$, which is labeled $2j$. 
 {If two edges between levels $n$ and $n+1$ have the same label $1$ or $2n$, they must have the same source (and different targets).}
{If two} edges between levels $n$ and $n+1$ have the same label {and different sources}, then for some $j = 1, ...,n-1$ one of these edges {has source $v(n, 2j)$,  the other has source $v(n, 2j +1)$, and the two edges have different targets and lie in $MC(j)$.}}\\

Proof of (2): 
	Two paths with the same sequence of edge labels cannot first differ into level $1$, because the edges out of $v(1,1)$ all have label $1$, while the edges out of $v(1,2)$ all have label $2$; so the paths first differ into level $j\geq 2$, arriving at different vertices at level $j$.
	Leaving level $j$, the two paths traverse edges with the same label and different sources. 
 {By part (1), both of these edges are in the same Morse component and have different targets.
Hence, both paths also traverse different edges with the same label and different sources into level $j+2$.  
 By repeated applications of part (1) we get that the two paths enter a Morse component at level $j$ and remain in that Morse component at all levels $j$ and higher.}

 Conversely, if a pair of paths $x$, $x'$ with the same sequence of edge labels enters a Morse component at level $j$ {and remains in that Morse component at all levels $j$ and higher}, then $x$ and $x'$ first differ at level $j$. 
 	This is because,
 from above, the pair cannot differ before entering a Morse component, so they meet different vertices at a first level $n>j$. 
 But since two edges in a Morse component with the same label have different sources, $n=j$. 
 In fact, $x$ and $x'$ do not ever meet the same vertex in that Morse component, 
 since the edges entering any vertex have distinct labels.

{Proof of (3):
If paths $x$ and $x'$ have the same sequence of edge labels,  
	{then by {(2)} 
	 $x$ and $x'$ enter a Morse component $MC(k)$ at the first level $j > k$ where they differ, and they will meet distinct vertices in $MC(k)$ at all levels $j$ and higher. }
	Hence, by Observation \ref{obs:GJ} (1), they will have the same $k$-basic blocks at levels $j$ and higher, whereas their $(k+1)$-basic blocks at level $j+1$ will differ. 
By Observation \ref{obs:GJ} (2), they will have the same ordinal path label from the root into all levels.} 
{Thus the dot will be in the same place in their identical $k$-basic blocks at all levels $k+1$ and higher: the pair will be $k$-equivalent but not $(k+1)$-equivalent.}

{Conversely, we show that any pair of $k$-equivalent paths $x$, $x'$ that are not $(k+1)$-equivalent enter $MC(k)$ at the first level where they differ and have the same sequence of edge labels.
By the definition of $k$-equivalence, such paths agree to level $k$ and their $k$-basic blocks at levels $k+1$ and higher are the same with the dot in the same place. 
In particular, $x$ and $x'$ have the same ordinal path label from level $k$ into any higher level and they agree to level $k$, so they have the same ordinal path label from the root into any level. 
Hence, by Observation \ref{obs:GJ} (2), they have the same sequence of edge labels. 
Denote by $j$ the first level at which $x$ and $x'$ differ (so that $j>k$).
Then, by {(2)}, for some $i < j$ both paths enter $MC(i)$ at level $j$ and are contained in that component at all higher levels. 
By the preceding paragraph, $x$ and $x'$ are $i$-equivalent but not $(i+1)$-equivalent.
Therefore,  $i=k$.}

{Proof of (4):  
Suppose that
 $k \geq 2$ and $x$, $x'$ is a depth $k$ pair.
 Then $x$ and $x'$ have the same $k$-coding and therefore the same $2$-coding.  
As shown later in Example \ref{ex:modGJ}, this means that $x$, $x'$ are $2$-equivalent, 
in other words, at all levels after level {$2$}, they have the same $2$-basic block with the dot in the same place.
{At any level $n \geq k$, the {$k$-basic block} $B_k(v)$ at any vertex $v$ factors via the $1$-block code mentioned {in Definition \ref{def:bb}} onto $B_2(v)$, and the result has the same length with the dot in the same position.}
It follows that the $k$-basic blocks at  $v_n(x)$  and $v_n(x')$ have equal length and the dot in the same position.
Since $x$ and $x'$ have the same $k$-coding, {then} their $k$-basic blocks at all levels {must} be the same.
Therefore, $x$ and $x'$ are $k$-equivalent. 
Furthermore, $x$ and $x'$ are not $(k+1)$-equivalent, since at some time in their orbits they follow different paths into level $(k+1)$ (by definition of depth $k$), whereas by Proposition \ref{prop:samecodings} $(k+1)$-equivalent paths must have the same $(k+1)$-coding. 
Hence, for all $k \geq 2$, any depth $k$ pair is $k$-equivalent but not $(k+1)$-equivalent.}

{Conversely, for all $k \geq 2$, any pair of paths that is $k$-equivalent has the same $k$-coding by Proposition \ref{prop:samecodings}. 
If that pair is not $(k+1)$-equivalent, then, it is not $i$-equivalent for any $i >k$. 
As just shown, this means it cannot be depth $i$ for any $i > k$. 
Therefore that pair is depth $k$.}
	\end{proof}

\begin{corollary}\label{cor:GJSNE}
	{The GJ example is standard nonexpansive}
	\end{corollary}
\begin{proof}
	{By Part {(4)} of the Proposition, given $k \geq 2$ two paths with the same sequence of edge labels that enter $MC(k)$ at the level where they first differ are depth $k$. 
		For example, the path $x$ that passes through the vertices $v(1,j)$ for {all} $j,1\leq j \leq k$, then $v(j,2k)$ for all $j>k$ and the path $y$ that passes through the vertices 
		$v(1,j)$ for $1\leq j \leq k$, then $v(j,2k+1)$ for all $j>k$, are depth $k$.}
	\end{proof}
	
	{The following observations are related to Example \ref{ex:modGJ}.
			\begin{obs}\label{obs:aperiodic}
				In the GJ example, the $2$-coding of every path is aperiodic (also sometimes called ``nonperiodic").
						\begin{proof}
								We will show that the forward coding of the minimal path $x_{\min}$ by vertices at level $2$, {$(v_2(T^jx), j \geq 0)$,} 
				is aperiodic. 
				{This implies that the $2$-coding of $x_{\min}$ is aperiodic, and hence, since the orbit of $x_{\min}$ is dense, the $2$-coding of every path in $X$ is aperiodic.}
				Fix a large $n \geq 3$.
				The idea is to reduce $C_2(v(n,1))$ to a long initial block of the famous Prouhet-Thue-Morse sequence, which is known to be aperiodic.\\
				\indent
				{For $n \geq 2$ we have}
				\be\label{eq:vertexcoding}
				C_{n-1}(v(n,1))=v(n-1,1)v(n-1,2)v(n-1,3) \dots v(n-1,2n-2), 
				\en
				but note that {for $n \geq 4$}
				\be
				C_{n-2}(v(n-1,3))=v(n-2,1)v(n-2,3)v(n-2,2)v(n-2,) \dots v(n-2,2n-4).
				\en
				This switch in order of adjacent symbols occurs at every level {$n-1 \geq 3$.}
				Working towards the expansion of $C_2(v(n,1))$ as a string on symbols $v(2,m), m=1,2,3,4$, in Equation \ref{eq:vertexcoding} replace each $v(n-1,i)$ by $C_{n-2}(v(n-1,i))$,
				then in the result (which is $C_{n-2}(v(n,1))$)
				replace each $v(n-2,i)$ by $C_{n-3}(v(n-2,i))$, etc., until we finally arrive at $C_2(v(n,1))$. 
				Note that, reading from left to right in any $C_{j}(v(n,i)), 2 \leq j<n$, from time to time the symbols $v(j,2)$ and $v(j,3)$ switch order. \\
				\indent
				We will now repeat this process of repeatedly expanding $C_{n-1}(v(n,1))$, deliberately losing some information at each step, 
				to produce for each $j=n-1, \dots, 2$ a string $\tilde C_j$ on the alphabet $\{v(j,2),v(j,3),0_j\}$.
				In Equation \ref{eq:vertexcoding}, 
				replace each $v(n-1,i)$ for $i$ not equal to $2$ or $3$ by $0_{n-1}$, arriving at a block $\tilde C_{n-1}$ on the alphabet $\{v(n-1,2),v(n-1,3),0_{n-1}\}$.\\
				\indent
				Then in $\tilde C_{n-1}$ replace each $v(n-1,2)$ by $C_{n-2}(v(n-1,2)$, each $v(n-1,3)$ by $C_{n-2}(v(n-1,3))$, each $0_{n-1}$ by $(0_{n-2})^{2n-4}$, 
				and finally each $v(n-2,i)$ for $i$ not equal to $2$ or $3$ by $0_{n-2}$, 
				arriving at a block $\tilde C_{n-2}$ on the alphabet $\{v(n-2,2),v(n-2,3),0_{n-2}\}$.
				Noting that $|C_{n-2}(v(n-1,i))|=2n-4$ for all $i$, we have formed a $1$-block factor $\tilde C_{n-2}$ of $C_{n-2}(v(n,1))$ on the alphabet $\{v(n-2,2), v(n-2,3),0_{n-2}\}$.\\
				\indent
				Continue analogously until in
				the end we have a block $\tilde C_2$ on the alphabet $\{v(2,2),v(2,3),0_2\}$ which is a symbol-by-symbol ($1$-block) factor of $C_2(v(n,1))$. \\
				\indent
							{If the forward coding of $x_{\min}$ by vertices at level $2$ were periodic, then  there would be a nonempty block $P$ on the symbols {$v(2,m) \, (m=1,2,3,4)$}, and a (possibly empty) prefix $Q$ of $P$, such that 
							for large enough $n$ we would have
							$C_2(v(n,1))=P^k Q$ for some $k \geq 2$.} 
				Then $\tilde C_2$ would have the form $\tilde P^k \tilde Q$ with $\tilde P, \tilde Q$ $1$-block images of $P,Q$.\\
				\indent
				But note that 
				{for each $j=n-1, \dots, 3$} a symbol $v(j,2)$ in $\tilde C_j$ expands 
				{to a block {$0_{j-1} v(j-1,2) v(j-1,3) 0_{j-1}^{2j-5}$ in $\tilde C_{j-1}$, while a symbol 
					$v(j,3)$ in $\tilde C_j$ expands 
					to a block $0_{j-1} v(j-1,3) v(j-1,2) 0_{j-1}^{2j-5}$ in $\tilde C_{j-1}$.}}
				Thus if we ignore {the symbol $0_{j-1}$ in each} $\tilde C_j$ we see substrings on $\{a=v(\cdot,2),b=v(\cdot,3)\}$ that expand according to the Prouhet-Thue-Morse (PTM) substitution $a \to ab, b \to ba$.\\
				\indent
				If $\tilde C_2$ were of the form $\tilde P^k \tilde Q$ as above, {with $n$ large enough that $k \geq 2$, then deleting the {symbol $0_2$ from} $\tilde P$ and $\tilde Q$ would present a long initial block of the PTM sequence in the form $P_0^k Q_0$ with $P_0,Q_0$ blocks on $\{a,b\}$, and $|P_0| \geq 1$ {(because $\tilde C_2$ contains symbols other than $0_2$)}.}
				But the PTM sequence $abba\,  baab \dots$ is aperiodic, in fact it cannot begin with $BB$ for any block $B$.
				Therefore the forward coding of $x_{\min}$ by vertices at level $2$ is not periodic, and hence the $2$-coding of $x_{\min}$ is not periodic.
			\end{proof}
\end{obs}
}

\begin{obs}\label{obs:recog}	
	{In the GJ example, identifying at every level $j>2$ the vertices other than $v(j,3)$ produces a sequence of morphisms that forms a recognizable family.}
	\end{obs}
	\begin{proof} 
		{Denote by $\mathcal V_j$ the set of vertices in our diagram at level $j$. 
			For each $j \geq 2$ define an alphabet $A_j=\{D_j,E_j\}$ and a 
			(many-to-one) { map} $\phi_j: \mathcal V_j \to A_j$ as follows. 
		Assign to $v(j,3)$ the symbol $D_j=\phi_j(v(j,3))$, and to each $v(j,i)$ for $i \neq 3$ assign the symbol $E_j=\phi_j(v(j,i))$.
			When we expand each symbol (vertex) $v(j,i)$ to $C_{j-1}(v(j,i))$ the effect on the {$\phi_j$-images} 
				{($D_j,E_j,D_{j-1},E_{j-1}$)
				at both levels produces the morphism 
				\be\label{eq:morphism}
{	E_{j} \to E_{j-1}^2D_{j-1}E_{j-1}^{2j-5}, \qquad D_{j} \to  E_{j-1}D_{j-1}E_{j-1}^{2j-4}}
	\en
	 {for all $j \geq 3$}, in the sense of concatenation of blocks.}\\
\indent
			{Denoting as usual by $A_j^+$ the set of nonempty words on the alphabet $A_j$}, {Equation (\ref{eq:morphism}) defines} {a sequence of morphisms $\tau_{j}: A_{j} \to A_{j-1}^+$}, as (for example) in \cite{Berthe2017}. 
		By keeping track of the position of $D_j$ in codings of paths {by images of vertices at level $j$ under $\phi_j$} one can prove directly that the sequence is {\em recognizable}, in the sense that every sequence on the alphabet $A_j$ has at most one {\em desubstitution}, or {\em factorization}, on the alphabet $A_{j+1}$.
		To see this, suppose that we are given a bisequence on the alphabet $A_j$.
		When we see a block $F(j,q)=D_j E_j^q D_j$, {since the block $D_{j+1}D_{j+1}$ does not appear in the coding of any path by the images of vertices at level $j+1$, it} must be the case that $F(j,q)$ is a subblock, in a uniquely determined position, of
		\be
		\begin{aligned}
		&E_{j+1} E_{j+1} = E_{j} E_j D_{j}E_{j}^{2j-3} \,\, E_{j} E_j D_{j}E_{j}^{2j-3}\quad\text{ if } q=2j-1,\\
		&E_{j+1} D_{j+1} = E_{j} E_j D_{j}E_{j}^{2j-3}  \,\, E_{j} D_j E_{j}^{2j-2}\quad\text{ if } q=2j-2, \text{ or }\\
		&D_{j+1} E_{j+1}= E_{j} D_j E_{j}^{2j-2} \,\,  E_{j} E_j D_{j}E_{j}^{2j-3}\quad\text{ if } q=2j.
		\end{aligned}
	\en
	Thus the coding of an orbit by $D_2,E_2$ determines its coding by $D_j,E_j$ for all $j > 2$.
		This also follows from \cite{Berthe2017}*{Theorem 5.1}, since each of these alphabets has only two elements and each infinite bisequence generated by the sequence of morphisms is aperiodic, by {the arguments in the proof of} Observation \ref{obs:aperiodic}.}
	\end{proof}
	\indent
	Note that the natural sequence of morphisms for the diagram defined by the coding of vertices at each level $j$ beyond some level $j_0$ by the vertices at level $j-1$ {\em cannot} be recognizable, because the GJ system is nonexpansive.
		(The coding of any path by vertices at some level would then determine the codings at all subsequent levels and hence the entire path.)	
		Moreover, the BV system corresponding to the sequence of morphisms $\tau_j$ is of finite topological rank (as well as expansive). In fact it has topological rank $2$: it cannot be $1$ (conjugate to an odometer) by \cite{FPS2017}*{Theorem 5.3} and Observation \ref{obs:aperiodic}, {since for every $n>1$} the $n$-coding of the orbit of any path is not periodic. \\

We show now how the GJ example can be modified to spoil the strict requirements of the definition of $SSNE$ to produce an example of a nonexpansive system that is not even $SNE$ and is not conjugate to any odometer. (By \cite{dm2008} any such example necessarily has unbounded width.) Further, this example is very well timed (see Definition \ref{def:WS}).

 {In \cite{dm2008}, Downarowicz and Maass introduced
		a handy way to visualize paths and their orbits in a $BV$ system by means of 3-sidedly infinite arrays of {\em $j$-symbols}}. 
		{Every vertex $v$ at level $j \geq 0$ has an associated $j$-symbol which is also labeled $v$. 
	{The $j$-symbol} is a finite rectangular matrix with $j+1$ rows consisting of subrectangles ($i$-symbols for $0 \leq i \leq j$), as described in \cite{dm2008}*{p. 741}.}
	{An array represents the entire orbit of a path and its $k$-codings for all $k \geq 0$.
		The path itself is indicated by an arrow pointing to the left edge of its time $0$ rectangle {at level $0$}, and this arrow indicates, by extending it vertically downward, all the other rectangles (vertices) through which the path passes at levels $n>0$.
		We call this extension the {\em vertical} that corresponds to the path. 
		Further development and use of these arrays can be seen in \cite{BezuglyiKwiatkowskiMedynets2009} and \cite{Hoynes2017}.}

 The family of arrays determined by a diagram has a type of consistency called ``agreeable" in \cite[Def. 3.2]{BezuglyiKwiatkowskiMedynets2009}: each $j$-symbol with a fixed name {$v$} has for its first {$j$} rows the same concatenation of $(j-1)$-symbols.
 Conversely, given an agreeable family of arrays, we can construct its unique associated $BV$ diagram (which may or may not be properly ordered).
{For each vertex $v$ at level $j$, the names, in order, of the {$(j-1)$-symbols} comprising the first {$j$} rows 
	of its associated $j$-symbol specify the edges connecting certain vertices at level $j-1$ to it in a certain order, possibly with repeats.} 
{This concatenation of $(j-1)$-symbols determines the $(j-1)$-basic block {$B_{j-1}(v)$} at the vertex v, since it lists in order the paths from the root to $v$.}
The Vershik map $T$ on the diagram corresponds to sliding each array one ``notch", i.e. level-$0$ rectangle width, to the right. 
Thus there is a natural correspondence between 
$BV$ systems, their diagrams, and agreeable families of arrays; in the following we deal with them interchangeably.

 \begin{example}\label{ex:modGJ}
 	 Denote the system in the GJ example \cite{GJ2000}*{Figure 4, p. 1699} by $X$.\\
  	\indent
 	 	In their proof in \cite{dm2008}, Downarowicz and Maass make use of a modification of the Bratteli diagram and hence of all arrays representing orbits of the system.  See \cite[pp. 743--744, Figure 4]{dm2008} and \cite[p. 204]{Hoynes2017}).
 	We will use a version of their splitting technique to ensure that for every $i \geq 1$ and every $j \geq i +1$  the two vertices at level $j$ in the $i$'th Morse component have different $1$-basic blocks.\\
 	\indent 	
 	For every $j \geq 2$ and every $i <j$, we replace the $j$-symbol $v_i = v(j, 2i)$ with two symbols $v_i'$ and $v_i''$ so that {$|B_{1}(v_i')| =  |B_1(v(j-1, 1))|$} and the concatenation of $B_{1}(v_i')$ and $B_{1}(v_i'')$ is $B_{1}(v_i)$.    
 	In particular,  $B_{1}(v_i')$ is a proper prefix of $B_{1}(v_i)$.
 	In the DM arrays, every occurrence of the $j$-symbol with label $v_i$ is then replaced by two $j$-symbols labeled $v_i'$ and $v_i''$ respectively. 
 	We leave the $j$-symbols for all other vertices at level $j$ unchanged, so their $1$-basic blocks are the same in the new diagram as they are in $X$. 
 	The result is that the left vertices in every Morse component of $X$ have been split into two vertices, whereas the right vertices remain intact.\\
 	\indent 	
 	Note that for every $j \geq 2$, the modifications at level $j$ extend vertical bars (i.e rectangle boundaries) at level $j-1$ in the original diagram for $X$ by only one level. None of these bars is extended further by a modification at level $j +1$. Hence, no new infinitely long vertical lines consisting entirely of rectangle boundaries are created.  
 	This means that after all modifications, the system represented by the diagram remains in the class of  properly ordered systems. 
 	Moreover, every orbit still meets every vertex, so the system is simple.
 	Also note that since this new system is conjugate to $X$, it is not conjugate to any odometer.\\
 	\indent
 	The new system has the property that at any level $j \geq 2$ {and $i = 1, ...,j-1$}, the $1$-basic block at the right vertex $v(j, 2i+1)$ of {the i'th} Morse component through level $j$ is longer than the $1$-basic blocks at the two new vertices.  
 	Hence, the right vertex cannot be $1$-equivalent to either of these new vertices. 
 	{We claim that for every $k\geq 1$ paths that were $k$-equivalent in the old diagram are no longer $k$-equivalent in the new diagram.  		
 		This is because $k$-equivalent paths in the old diagram are also $1$-equivalent in that diagram. 
 		By Proposition \ref{prop:GJSNE}, any pair of paths that are $1$-equivalent in the original diagram {are contained in a Morse component at all levels after which they first differ}, so they can no longer be $1$-equivalent in the new diagram.
 	Then since $k$-equivalence implies $1$-equivalence, the paths are no longer $k$-equivalent in the new diagram.
 	}\\
   	\indent 	
 	{The $1$-coding of every path in $X$ is periodic with period length 2. 
 	However, at every level $j >2$ {the $2$-basic block at $v(j,3)$ differs from the $2$-basic blocks at at all other vertices,} and as a result the $2$-coding is not periodic (see Observation \ref{obs:aperiodic}).  
 	We will exploit this property of $2$-basic blocks to show that no new $2$-equivalent paths are created by our modifications. It then follows that the new system has no $2$-equivalent paths, hence is not $SNE$.}\\
 \indent 
      	By {Prop \ref{prop:samecodings}}, any $2$-equivalent pair in the new system has the same new $2$-coding.  In fact, it is the image of a pair in $X$ that has the same old $2$-coding.
 	This is because there is an invertible mapping between the old and the new codings.  
 	Specifically, in both the old and the new system each path in $X$ is represented by a vertical in its array. 
 	Row 2 of the array displays the $2$-coding of the path and the placement of the dot is specified by the vertical. 
 	The splitting of rectangles and relabeling at level 2 is reversible and does not change the position of the vertical.  
 	 	So a pair of $2$-equivalent paths in the new diagram is the image of a pair in the diagram for $X$ with the same $2$-coding. 
 	 	We now show that any such pair {in $X$} is $2$-equivalent.\\
  \indent 	  
  	Let $x$, $x'$ be paths in the diagram for $X$ with the same $2$-coding.
 	Using the preceding observations about $D_j$ and $E_j$ {(see Observation \ref{obs:recog}),} we argue inductively that for every $j \geq 3$, any time the orbit of $x$ is minimal from the root to $v(j,3)$, the orbit of $x'$ is as well (and vice versa). 
 	{Then because all $j$-basic blocks at level $j$ have the same length ($\dim v(j,1)$)}, any time the orbit of one of these paths changes vertices at level $j$, the other one does as well.  
 	Since $x$ and $x'$ have the same 2-coding, it follows that they must be $2$-equivalent at level $j$ (i.e not only do they have the same $2$-basic block at level $j$ but their dot is in the same place).\\
 	\indent
  	{First note that, since $x$ and $x'$ have the same $2$-coding, their orbits must meet {$v(2,3)$} always at the same time; in particular,  
 	any time the orbit of one of these paths is minimal from the root into $v(2,3)$ the orbit of the other is as well.\\  
 	\indent
 	Next fix $j \geq 2$ and assume that the orbits of $x$ and $x'$ are minimal into $v(j,3)$ always at the same time.}  
 	Since the number of paths is the same from the root into any vertex at level $j$, these orbits must be minimal into each vertex at level $j$ always at the same time.\\
 	\indent
 	{Now suppose that for some $m$, $T^mx$ is minimal from the root into $v(j+1,3)$, so that  $\phi_{j+1}(v_{j+1}(T^mx))=D_{j+1}$ and $T^{m+\dim (v(j,1))}x$ is minimal to $v(j,3)$, {which maps under $\phi_j$ to $D_j$}.
 		We claim that {since the orbits of $x$ and $x'$ always hit $v(j,3)$ at the same time,} we cannot have 
 	$v_{j+1}(T^mx')=v(j+1,i)$ for some $i \neq 3$, i.e. we cannot have $\phi_{j+1}(v_{j+1}(T^mx'))=E_{j+1}$ rather than $D_{j+1}$.
 	This follows from looking at the {\em next} vertices at level $j+1$ hit by the orbits of $x$ and $x'$. 
 	Since the block $v(j+1,3) v(j+1,3)$ cannot appear in the coding of any path by vertices at level $j+1$, the orbit of $x$ next hits $v(j+1,i)$ for some $i \neq 3$, which has image $E_{j+1}$, while the orbit of $x'$ next hits a vertex with image either $E_{j+1}$ or $D_{j+1}$.
 	As seen in {the proof of} Observation \ref{obs:recog},
 	in the block $D_{j+1} E_{j+1}$ (on symbols $D_j,E_j$) in the coding of the orbit of $x$ (by images of vertices under $\phi_{j}$), consecutive appearances of $D_j$ are separated by a distance $2j$, while in the two possible blocks $E_{j+1} E_{j+1}$ and $E_{j+1} D_{j+1}$ in the coding of the orbit of $x'$ the consecutive appearances of $D_j$ are separated by distance either $2j-1$ or $2j-2$. 
 	Therefore $v_{j+1}(T^mx')=v(j+1,3)$.}\\
 	\indent
 	 Since $T^{m+\dim (v(j,1))}x$ is minimal to $v(j,3)$ and the orbits of $x$ and $x'$ always hit $v(j,3)$ at the same time, we must have that $T^{m+\dim (v(j,1))}x'$ is also minimal to $v(j,3)$.  
 	 {Since the dot for both $\phi_j(T^{m+\dim (v(j,1))}x)$ and $\phi_j(T^{m+\dim (v(j,1))}x')$ is at the beginning of an appearance of $D_j$, and $D_j$ appears only once in each $\phi_j(v(j+1,3))$,
 	  applying $T^{-\dim ((v(j,1))}$  
 	shows that both $T^mx$ and $T^mx'$ are minimal from the root into $v(j+1,3)$.} 
   	 \end{example}

\section{Well timed and untimed systems}\label{sec:welltimed}
In this section we define several classes of nonexpansive $BV$ systems ($W,W_0,DM2$, $H2, U,U_0,U_1$, and $U_2$) according to various possibilities for the existence of pairs of paths with cuts. 
	{Recall that all systems under consideration are nonexpansive, properly ordered, and simple.}
	It will turn out that every {such} {system} is conjugate to {a system in} exactly {one of $W$  (``well timed") or $U$ (``untimed").}
 
{\begin{definition} 
		{For any class $S$ of systems, we will denote the class of systems conjugate to some system in $S$ by $CS$:
			\be
			CS=\{X: \text{there is } Y \in S \text{ such that } X \text{ is conjugate to }Y\}.
			\en
			We denote the class of systems not conjugate to any system in $S$ by $NCS$, and the class of systems not in $S$ by $\neg S$.
			Note that it is not true that if $X,Y \in CS$ then $X$ must be conjugate to $Y$.}
\end{definition}}

In their proof that every bounded width system is either expansive or conjugate to an odometer, Downarowicz and Maass \cite{dm2008} considered a class of systems that they called {\em Case (2)}, which we denote here by {\em $DM2$}. 
In \cite{Hoynes2017} Hoynes' Case (2) is a slightly weaker condition, which we call $H2$, {apparently} still sufficient for the proof to succeed. 
Downarowicz and Maass as well as Hoynes actually used stronger properties, {opposite to well timed}, which here we call {\em very untimed ($U_0$)} and $U_2$. 
Here is a list of some relevant classes of systems, obtained by varying quantifiers. 

\begin{definition}\label{def:WS}
	We say that a depth $k$ pair of paths {\em has long cuts} if for every $n>k$ the pair has an $n$ cut.
\end{definition}

\begin{definition}\label{def:classes}
{We define the following classes of systems:}

		\noindent
	(1) 		We say that a system is {\em well timed} if for every $k \geq 1$ for every $j>k$ there is a depth $k$ pair with a $j$ cut. 
		We denote the class of well timed systems by $W$.
	
\noindent(2) 		We say that a system is {\em very well timed} if for every $k \geq 1$ there exists a depth $k$ pair with long cuts. 
		We denote the class of very well timed systems by $W_0$.
	
	\noindent
	(3) $DM2$: For infinitely many $k$ there is a $j(k)>k$ such that no depth $k$ pair has a $j(k)$ cut. 
	The smallest such $j(k)$ is called a {\em $k$ cutoff}.
	
	\noindent
	(4) $H2$: For infinitely many $k$ for every depth $k$ pair $x,x'$ there is a $j(k,x,x') > k$ such that $x,x'$ has no $j(k,x,x')$ cut. 
	The smallest such $j(k,x,x')$ could be called a {\em $k$ pair cutoff for the pair $x,x'$.}
		
	\noindent
	(5) $U$ (untimed): For every $k$ there is a $k$ cutoff. (I.e., for every $k \geq 1$ there is a $j(k)>k$ such that no depth $k$ pair has a $j(k)$ cut.)
	
		\noindent
	{(6) $U_0$ (very untimed): For every $k$ no depth $k$ pair has a $k+1$ cut. (I.e., for every $k$, $k+1$ is a $k$ cutoff).}
	
	\noindent
	(7) $U_2$: For every $k$ there is a depth $k$ pair with no $k+1$ cut.
	
	\noindent
	(8) $U_1$: For infinitely many $k$ there is a depth $k$ pair $x,x'$ and there is a $j(k,x,x')>k$ such that the pair $x,x'$ has no $j(k,x,x')$ cut. (I.e., for infinitely many $k$ there is a depth $k$ pair with a pair cutoff.)
\end{definition}

{
Recall that the GJ example is $SNE$ and is very well timed.} 
 The modified GJ example (Example \ref{ex:modGJ}) is not $SNE$ but it is very well timed, because 
the only changes we made to the DM arrays were to add additional vertical bars, which will not destroy the existing cuts.

\begin{remark}\label{rem:classes}
	In \cite{dm2008} the proof of the main theorem in Case (2) begins by telescoping any system in $DM2$ so that the result is in $U_0$. Hoynes \cite{Hoynes2017}*{Remark 4.4} does not see why this should always be possible, {but suggests that changing the universal quantifier to existential, i.e. replacing $DM2$ with $U_1$, does allow one to telescope such a system to one in $U_2$, and that should be enough to let the proof proceed.}\\
		\indent 
	{The proof in \cite{Hoynes2017} assumes $H2$, telescopes so that for every $k \geq 1$ there is a depth $k$ pair, and then telescopes to obtain a system in $U_2$. 
		The proof of Sublemma 4.1, though, applies $H2$ to possibly ineligible pairs $y_i,y_{i'}$, {because $H2$ is not closed under telescoping}.
		Because the telescoped system is in $U_2$, for every $i$ there is a pair $x_i,x_i'$ with no $i+1$ cut. 
		But the pair $y_i,y_{i'}$ of the proof could be depth $i'$ and without a cutoff, if it were the image under the telescoping of a pair of a depth other than one of the infinitely many ``good" $k$ in the definition of $H2$ (see Proposition \ref{prop:newlemma}).}\\
	\indent
		Indeed, in Example \ref{ex:DM2WW} we present a system that is in both classes $DM2$ {(vacuously)} and $WW$ {(see Definition \ref{def:WW}, below)} {and for which telescoping to any strictly increasing sequence of levels takes it out of the class $H2$ (and hence out of $DM2$).}
    Such a system cannot be telescoped into $U_0$, because then it would be in both $CU$ and $WW=CW$ {(see Proposition \ref{prop:CW})}, but by Theorem \ref{thm:CWimpliesNCU} these classes are disjoint.\\
			\indent
			We think that Sublemma 4.1 of \cite{Hoynes2017} can be proved as follows.
			 {Assuming that $H2$ is satisfied nonvacuously, 
			  let $i_0$ and all the other $i$’s mentioned in the argument be good $k$’s according to the definition of $H2$.}
			  They may not fill up an interval in the integers, but given $L$ one can pick a sequence of them of length $L$ and proceed to write the same argument, being careful to choose the {pairs $x_i,x_i'$ so that their cutoffs $j(i,i')$ interleave the levels with good $k$'s.} 			
				\end{remark}

	\begin{remark}\label{rem:untimedtel}
		(1) { $U_0 \subset U \subset DM2 \subset H2$, and $U_2 \subset U_1$.}\\
	\indent	
	(2) {Each of the classes $W,W_0,U,U_0$ is closed under telescoping.} 
	{We will later {(in Remark \ref{rem:telw})} provide a proof of this for $W$ and $W_0$.}
	 To see that the very untimed property persists under telescoping, note that if in a telescoping to levels {$\{ n_l, l \geq 0\}$} of a very untimed system we found a depth $j$ pair with a $j+1$ cut, that pair would correspond to a pair in the original system of some depth $k \geq n_j$ with an $n_{j+1}$ cut, and hence with a $k+1$ cut{---cf. Proposition \ref{prop:newlemma}, (3) and (4)}.
  	\end{remark}

\begin{example}\label{ex:odoms}
		Every odometer presented with one vertex at every level is in $U_0$ (very untimed).
	 To see this, suppose that $x,x'$ are depth $k \geq 1$, so that at some time $m$ in their orbits they follow different edges from level $k$ to level $k+1$, {i.e. $(T^mx)_k \neq (T^mx')_k$.}\\
	    \indent
	    Because $(T^jx)_k$ and $(T^jx')_k$, $j \in \Z$, follow the same periodic sequence of edges as $j$ varies, we have that, for all $j \in \Z$, $(T^jx)_k \neq (T^jx')_k$. In particular, $x$ and $x'$ cannot have a $k+1$ cut, since this would require that for some $j$ they both follow the unique minimal edge from level $k$ to level $k+1$.\\
	    \indent
	    Note that in this example for every $k \geq 1$ there is a depth $k$ pair, but that is not a requirement in the definition of $U_0$.\\
	    \indent 
	    We used here a general principle that applies in any system: If level $k+1$ is strongly uniformly ordered with respect to level $k$ (see Definition \ref{def:suo}) and $x,x'$ follow edges with different ordinal labels from level $k$ to level $k+1$, so do $T^jx$ and $T^jx'$ for all $j \in \Z$.
	    
\indent
{On the other hand, every odometer can be presented (up to conjugacy) with at least two vertices per level after the root
	and all levels strongly uniformly ordered.
	 We claim that for such a system for every $k \geq 1$ there is a depth $k$ pair with a $k+1$ cut, 
	but no depth $k$ pair can have a $k+2$ cut.}
	{Thus every such system is in $U \setminus U_0$, with a $k$ cutoff of $k+2$ for every $k \geq 1$.}\\
	\indent
	At each level $n \geq 1$, the system has vertices $v(n,1), \dots, v(n,q_n)$ for some $q_n \geq 2$.
Edges with source $v(n,i)$ have ordinal label $i$.\\
\indent
Given $k \geq 1$, let $x$ be a path that is minimal from the root to $v(k+1,1)$, and let $x'$ be a path that is minimal from the root to $v(k,1)$ at level $k$ and then follows the edge (labeled $1$) to $v(k+1,2)$ at level $k+1$. 
{Because $T^jx,T^jx'$ are at strongly uniformly ordered vertices at level $k+1$ for all $j \in \Z$, 
they follow edges with the same ordinal label from level $k$ to level $k+1$}.
{Thus
$x,x'$ have the same $k$-coding and hence the pair is depth $k$. 
The paths $x,x'$ also follow minimal edges to level $k+1$, so they have a $k+1$ cut. 
Thus the system is not in $U_0$.}\\
\indent
We show now that if $x,x'$ is a depth $k$ pair, then it cannot have a $k+2$ cut, so the system is in $U$, with $k$ cutoff equal to $k+2$.
For suppose that $x,x'$ is a depth $k$ pair. 
Applying a power of $T$ if necessary, we may assume that these paths 
follow different edges from level $k$ to level $k+1$. 
{Because level $k+1$ is strongly uniformly ordered {with respect to level $k$},} for all $j \in \Z$ the paths $T^jx$ and $T^jx'$ follow different edges from level $k$ to level $k+1$, and hence they are at different vertices at level $k+1$:
\be
v_{k+1}(T^jx) \neq v_{k+1}(T^jx').
\en
Thus the edges downward from these vertices to level $k+2$ always have different ordinal labels, precluding existence of a $k+2$ cut.
 \end{example}

	\begin{comment}
\indent
Question: Is every system conjugate to an odometer (in particular every odometer---def might be that all levels are uniformly ordered) in $U$? (It's in $CU_0 \subset CU$.)\\
\indent
We claim that it's in a related class: There are infinitely many levels $\{n_j\}$ such that no depth $n_j-1$ pair can have an $n_{j+1}$ cut. (This is a little bit like $U_1$.) \\
\indent
Any system conjugate to an odometer has a telescoping with infinitely many uniformly ordered levels \cite{FPS2017}. 
Suppose that $n_j$ and $n_{j+1}$ are two levels of the telescoping, so that level $n_{j+1}$ is uniformly ordered with respect to level $n_j$. 
 By an extension of the above principle, no depth $n_j-1$ pair can have an $n_{j+1}$ cut. 
 Because if $x,x'$ are depth $n_j-1$, then there is $m \in \Z$ such that they are at different vertices at level $n_j$. 
 Since the coding of each vertex at level $n_{j+1}$ in terms of vertices at level $n_j$ is, up to a shift, the same periodic sequence, $T^mx$ and $T^mx'$ can never be at the same vertex at level $n_j$, for any $m \in \Z$. 
 So for no $m \in \Z$ can $T^mx$ and $T^mx'$ follow only minimal edges from the root to level $n_{j+1}$, as would be required by their having an $n_{j+1}$ cut, since every vertex on level $n_{j+1}$ has the same source for its minimal incoming path from level $n_j$.\\
\indent
But we don't know that {\em there exist} depth $n_j-1$ pairs in the odometer or system conjugate to an odometer.}
\end{comment}

  We aim to show that $CW \subset  NCU$ and $NCW \subset CU$, so that the family of simple, perfectly ordered nonexpansive $BV$ systems is the disjoint union of $CW$ and $CU$;
   \be
{ W_0 \subset W \subset CW = NCU \subset NCU_0.}
  \en
  For this purpose we need to know how pairs of some depth and cuts in a system relate to those in a telescoping of that system. 
  {If $\tilde X$ is a telescoping of $X$, we will call $X$ a {\em lift} of $\tilde X$.
  The following Proposition says, informally, that pairs of some depth in one of $X,\tilde X$ 
  telescope (lift) to pairs of a related depth in the other,
  as do cuts for those pairs.
  {One consequence is that nonexistence of cutoffs is preserved under telescoping and lifts.}
   
\begin{proposition}\label{prop:newlemma}
 	Let $(X,T)$ be a system and $(\tilde X, \tilde T)$ be another system obtained by telescoping $X$. The following statements hold for every $k \geq 1:$\\
 	(1) There exists an $\tilde i(k) \leq k$ such that the image of any depth $k$ pair in $X$ under the telescoping is depth $\tilde i(k)$ in $\tilde X$.
	Furthermore, $\tilde i(k) \rightarrow \infty$ {as {$k \rightarrow \infty$.}}\\
	 	(2) For all sufficiently large $j$ there exists $\tilde J(j) \leq j$ such that if a depth $k$ {pair in $X$ has a $j$ cut then the depth $\tilde i(k)$ image} of that pair in $\tilde X$ has a $\tilde J(j)$ cut.
	 Furthermore, $\tilde J(j) \rightarrow \infty$ as $j \rightarrow \infty$.\\
 	(3) For every depth $k$ pair {$\tilde x^{(k)}, \tilde y^{(k)}$ in $\tilde X$ there exists an $i(\tilde x^{(k)}, \tilde y^{(k)}) \geq k$ such that $\tilde x^{(k)}, \tilde y^{(k)}$ is the image under the telescoping of a depth $i(\tilde x^{(k)}, \tilde y^{(k)})$ pair in $X$.}\\
 	(4)  For any $j > k$ there exists a $J(j)\geq j$ such that {if $\tilde x^{(k)}, \tilde y^{(k)}$ is a depth $k$ pair in $\tilde X$ with a $j$ cut, then that pair is the image under the telescoping of a depth $i(\tilde x^{(k)}, \tilde y^{(k)})$ pair in $X$ with a $J(j)$ cut.}
	% Furthermore, $J(j) \rightarrow \infty$ as $j \rightarrow \infty$.
  \end{proposition}
  
 \begin{proof} 
 	Suppose that $\tilde X$ is a telescoping of $X$ to levels $(n_l,l \geq0)$ {(with the root being at level $n_0=0$).} 
 	
 	 	Proof of (1) and (2): \\
				If $k<n_1$, then after telescoping from the root to level $n_1$, the image of any pair of paths that is depth $k$ in $(X, T)$ is a pair of paths with different 1-codings in $(\tilde X, \tilde T)$, {i.e., $\tilde i(k)=0$}.
		  	 	{So assume $k \geq n_1.$  In other words, there exists $l_k \geq 1$ such that $n_{l_k}\leq k < k+1 \leq n_{l_k+1}$.} 
		  	 	{In particular, $l_k \to \infty$ as $k \to \infty$.}

 	Let $x^{(k)}, y^{(k)}$ be a depth $k$ pair of paths in $X$. 
 	Since these paths agree to level $k$ and differ at level $k+1$, they have the same $n_{l_k}$ coding, but not the same $n_{l_k + 1}$ coding. 
 	This means row $n_{l_k}$ in the array {of $j$-symbols} for $x^{(k)}$ is identical to row $n_{l_k}$ in the array for $y^{(k)}$ and is met in the same position by the verticals 
 	({see the discussion preceding Example \ref{ex:modGJ}}) for both paths, whereas the same is not true for row $n_{l_k + 1}$.
 	 For every $l \geq 1$, the telescoping removes all rows of the arrays between rows $n_l$ and $n_{l+1}$, so that what used to be row $n_l$ becomes row $l$. 
 	 Afterwards, row $l_k$ in the array for the image of $x^{(k)}$ is identical to row $l_k$ in the array for the image of $y^{(k)}$ and is met in the same position by the verticals for both paths, whereas the same is not true for row $l_{k+1}$.  
 	 Hence, the image of the pair $x^{(k)}, y^{(k)}$ under the telescoping is depth $l_k$.  
 	 Letting $\tilde i(k) = l_k$, (1) is proved.

 	{Suppose that $j>k$ and the paths $x^{(k)}$ and $y^{(k)}$ have a $j$ cut. 
 		This cut appears as a pair of vertical segments in the arrays for the two paths that begin in the same position in row 0, end in row $j$, and consist entirely of rectangle boundaries.  
 	Find $l_j$ such that $n_{l_j}\leq j < n_{l_j+1}$.} 
 	 	After removing the rows for the telescoping, these vertical segments in the arrays for $x^{(k)}$ and $y^{(k)}$ extend from level $0$ to level $l_j$ and still consist entirely of rectangle boundaries.  
 	Hence both represent minimal paths in $\tilde X$ from the root to level $l_j$.  
 	It follows that the image of $x^{(k)}, y^{(k)}$ has an $l_j$ cut.  
 	Letting $\tilde J(j) = l_j$, (2) is proved.
 	
 	Proof of (3) and (4): \\ 	
 	Now let $\tilde x^{(k)}, \tilde y^{(k)}$ be a depth $k$ pair of paths in $\tilde X$.   
 	This means row $k$ in the array for $\tilde x^{(k)}$ is identical to row $k$ in the array for $\tilde y^{(k)}$ and is met in the same position by the verticals for both paths, whereas the same is not true for row $k+1$.  
 	In the original diagram, there is a pair of paths whose image under the telescoping is $\tilde x^{(k)}, \tilde y^{(k)}$.  
 	If we {restore} the rows in their respective arrays that were removed by the telescoping, rows $k$ and $k+1$ in the arrays for $\tilde x^{(k)}$ and $\tilde y^{(k)}$ become rows $n_k$ and $n_{k+1}$.  
 	{Hence, for some $i(\tilde x^{(k)}, \tilde y^{(k)}) \in [n_k, n_{k+1})$,} it must be the case that these arrays are now the same 
 	{from row $0$ to} row $i(\tilde x^{(k)}, \tilde y^{(k)})$ with the vertical in the same position, and that the same is not true for row $i(\tilde x^{(k)}, \tilde y^{(k)})+1$.   
 	It follows that $\tilde x^{(k)}, \tilde y^{(k)}$ is the image of a depth $i(\tilde x^{(k)}, \tilde y^{(k)})$ pair of paths in $X$ under the telescoping. This proves (3).
 	
 	For (4), it is important to note that when we reinsert rows into an array for $\tilde X$ that were removed during the telescoping to get an array for some path in $X$, any vertical segment bounding a rectangle in row $l \geq 1$ of the former becomes a rectangle boundary in row $n_l$ of the new array.   
 	Moreover, the vertical line containing this rectangle boundary at level $n_l$ must include rectangle boundaries from row $n_l$ all the way up to row 0.   
 	So any $j$ cut for a depth $k$ pair in $\tilde X$ just becomes longer when we reinsert {into their respective arrays rows} that were removed by the telescoping.  
 	Specifically, any $j$ cut for a depth $k$ pair of paths $\tilde x^{(k)}, \tilde y^{(k)}$ in $\tilde X$ {corresponds to} an $n_j$ cut for {the} preimage of $\tilde x^{(k)}, \tilde y^{(k)}$ before the telescoping. 
 	Letting $J(j) = n_j$, (4) is proved.
 	 \end{proof}
	 
	 {Some of the systems that we shall encounter while proving our main results will have the 
	 	property encapsulated in the following definition.
	 	\begin{definition}\label{def:WW}
	 	   We say that a system is \emph{weakly well timed} 
	 	 if for infinitely many $k$ for every $j>k$ there is a depth $k$ pair with a $j$ cut.
	 	 $WW$ denotes the class of systems with this property.
	 	 \end{definition}
 	 }
By definition, $W \subset \neg DM2 \subset WW$.
We will want to know what happens to the well timed and $WW$ properties under microscoping and telescoping. 

{It is not necessarily the case that if $(\tilde X, \tilde T)$ is {a telescoping of $(X,T)$ and is well timed,} then $(X, T)$ is well timed.
For example, suppose we telescope $X$ to even levels.
It could happen that  $(X, T)$ has no pairs with odd depth, and yet for every {$k \geq 1$ and every} $j >k$ there exists a depth $k$ pair in $(\tilde X, \tilde T)$ with a $j$ cut whose lift to $(X, T)$ is depth $2k$.  
In this case, $(\tilde X, \tilde T)$  is well timed while $(X, T)$ is not.}

\begin{lemma}\label{lem:lifts}
Let $(X, T)$ be a system and $(\tilde X, \tilde T)$ be another system obtained by telescoping $X$ to levels {$(n_l, l \geq 0)$.}  \\
(1) If $(X, T) \in W$, then $(\tilde X, \tilde T) \in W$.  Likewise, if $(X, T) \in W_0$, then $(\tilde X, \tilde T) \in W_0$. \\
(2) If $(X, T) \in WW$, then $(\tilde X, \tilde T) \in WW$.\\
(3)  If $(\tilde X, \tilde T) \in WW$, then $(X, T) \in WW$.
   \end{lemma}
\begin{proof}  
 Proof of (1): \\ 
Suppose $(X, T)$ is well timed.  
Fix $k \geq 1$ and let $j>k$. 
There exists in $X$ a depth $n_k$ pair with an $n_j$ cut.  
{As shown in the proof of Proposition \ref{prop:newlemma}, parts (1) and (2), after}
   the telescoping this yields a depth $k$ pair in $\tilde X$ with a $j$ cut.
 % Specifically, in the proof of  Proposition \ref{prop:newlemma}, {in part (1), ${l_k} = k$, 
 % and in the proof of $(2)$, ${n_j} = j$.}
  
 If $(X, T)$ is very well timed, then for every $k \geq 1$ there exists a depth $n_k$ pair with an $n_j$ cut for all $j >k$.  The above argument shows that this yields a depth $k$ pair with long cuts after the telescoping. 

Proof of (2):\\
 Suppose $(X,T) \in WW$. 
In other words, there exists an increasing sequence $(m_k)$ such that for every $k$ and every $j>m_k$ there is a depth $m_k$ pair in $(X, T)$ with a $j$ cut. 

By Proposition \ref{prop:newlemma} (1), for every $k \geq 1$ there exists $\tilde i(m_k) \leq m_k$ such that the image of any depth $m_k$ pair in $X$ under the telescoping is depth $\tilde i(m_k)$ in $\tilde X$. 
Let $\tilde m_k = \tilde i(m_k)$. 
Since $\tilde i(k) \rightarrow \infty$ as $j \rightarrow \infty$, the sequence $(\tilde m_k)$ is increasing.

Given $k \geq 1$ and $j >m_k$, find in $(X, T)$ a depth $m_k$ pair with a $j$ cut.   
By Proposition \ref{prop:newlemma} (2), there exists a $\tilde J(j) \leq j$  such that the depth $\tilde m_k$ image of this pair in $(\tilde X, \tilde T)$ has a  $\tilde J(j)$  cut. 
Since $\tilde J(j) \rightarrow \infty$ as $j \rightarrow \infty$, it follows that $(\tilde X, \tilde T) \in WW$. 

Proof of (3):\\
Suppose that  $(\tilde X, \tilde T) \in WW$.
There exists an increasing sequence $(\tilde m_k)$ such that for every $k$ and every $j> \tilde m_k$ there is a depth $\tilde m_k$ pair $\tilde x^{({\tilde m_k})},  \tilde y^{({\tilde m_k})}$ in $\tilde X$ with a $j$ cut. 

Fix $k \geq 1$.
As we vary $j$, the pair $\tilde x^{({\tilde m_k})},  \tilde y^{({\tilde m_k})}$ may change, and hence $i(\tilde x^{({\tilde m_k})},  \tilde y^{({\tilde m_k})})$ may vary.  
 However, it was shown in the proof of  Proposition \ref{prop:newlemma} (3) that for every depth $\tilde m_k$ pair $\tilde x^{({\tilde m_k})}, y^{({\tilde m_k})}$, {we have $n_{\tilde m_k} \leq  i(\tilde x^{({\tilde m_k})},  \tilde y^{({\tilde m_k})}) < n_{\tilde m_{k}+1}$}. 
 {Hence, there exists {$m_k \in [n_{\tilde m_k}, n_{\tilde m_{k+1}})$} such that for infinitely many $j>\tilde m_k$, the corresponding pair $\tilde x^{({\tilde m_k})},  \tilde y^{({\tilde m_k})}$  lifts to a depth $m_k$ pair, which {by Proposition \ref{prop:newlemma} (4)}  has a $J(j)>j$ cut.} 
 	\begin{comment}  
So {by Proposition \ref{prop:newlemma} (4),} for every $k \geq 1$ there exists {$m_k \in [n_{\tilde m_k}, n_{\tilde m_{k+1}})$} such that for infinitely many $j>\tilde m_k$, {there is a $J(j) \geq j$ such that} some depth $\tilde m_k$ pair with a $j$ cut in $\tilde X$ lifts to a depth $m_k$ pair in $X$ with a $J(j)$ cut. 
\end{comment}  
{Since $J(j) \rightarrow \infty$ as $j \rightarrow \infty$} (and the $m_k$ are all distinct), it follows that $(X, T) \in WW$.
\end{proof}

 \begin{proposition}\label{prop:CW}
   For a system $(X,T)$ the following statements are equivalent:\\
   (1) $(X, T)$ is weakly well timed.\\
   (2) $(X,T)$ has a telescoping that is well timed.\\
   (3) $(X,T)$ is conjugate to a well timed system.\\
   {Thus $WW=CW$.}
   \end{proposition}
\begin{proof} 

We show first that (1) implies (2).
Suppose there are infinitely many $k$ for which for every $j>k$ there exists a depth $k$ pair in $X$ with a $j$ cut.  
In other words, there exists an increasing sequence $(m_k)$ such that for all $k$ and for every $j > m_k$ there is a depth $m_k$ pair in $X$ with a $j$ cut. 
Let $\tilde X$ be the telescoping of $X$ to levels $(m_k)$.   

Given $k \geq 1$ and $j >m_k$, find in $(X, T)$ a depth $m_k$ pair with a $j$ cut.   
By Proposition \ref{prop:newlemma} {(1) and (2)}, the image of this pair is depth $k$ with a $\tilde J(j)$ cut. Since $\tilde J(j) \rightarrow \infty$ as $j \rightarrow \infty$, it follows that $\tilde X$ is well timed.

That (2) implies (3) is clear, since conjugacy of $BV$ systems is the equivalence relation that corresponds to the one for diagrams that is generated by telescoping and isomorphism \cite[Section 4]{HPS1992}.  
So if $(X,T)$ has a telescoping that is well timed then it is conjugate to that well timed system. 

 To prove that (3) implies (1), suppose that $(X,T)$ is conjugate to a well timed system $(Y,S)$. 
As remarked above, 
 then there are a system $Z$ and telescopings $\tilde X$ of $X$ and $\tilde Y$ of $Y$ such that $Z$ telescopes on even levels $0,2,4,\dots$ to $\tilde X$ and on odd levels $0,1,3, \dots$ to $\tilde Y$. 
By Lemma \ref{lem:lifts} (1), any telescoping of a well timed system is also well timed. So we may assume that $\tilde Y = Y$. 

Since $W \subset WW$, $Y$ is also weakly well timed.
Hence by Lemma \ref{lem:lifts} (3), $Z \in WW$. 
Next consider the telescoping of $Z$ to $\tilde X$. By Lemma \ref{lem:lifts} (2), $(\tilde X, \tilde T) \in WW$.

Applying part (3) of the lemma again, we conclude that $(X, T) \in WW$.
 \end{proof}
With appropriate adjustments the foregoing Proposition and its proof adapt to very well timed systems.
 \begin{proposition}\label{prop:CW0}
	For a system $(X,T)$ the following statements are equivalent:\\
	(1) There are infinitely many $k$ for which there exists a depth $k$ pair with long cuts.\\
	(2) $(X,T)$ has a telescoping that is very well timed.\\
	(3) $(X,T)$ is conjugate to a very well timed system.
\end{proposition}

\begin{remark}\label{rem:telw}
{Part (3) of Proposition \ref{prop:CW0} shows that parts (1) and (2) persist under telescoping.}  
\end{remark}

 \begin{example}\label{ex:DM2WW}
 {As promised in Remark \ref{rem:classes}, we show that $WW \cap DM2 \neq \emptyset$.
 	We modify {the diagram for the GJ example $X$} by changing every other {(even-numbered)} Morse component so that all its vertices are uniformly ordered: 
 	for all $j \geq 4$ and all even $i$ such that $1 \leq i < j-1$ we change the ordering at $v(j,2i +1)$
 	so that it is left to right. 
 	Let $Y$ denote the new diagram {and system}, {with the same vertices $v(j,i)$ and ``Morse components" $MC(k)$ as in $X$.}
 	Then for every odd $k$, there is still a depth $k$ pair with long cuts (these are just the same as they were in the GJ example, see Proposition \ref{prop:GJSNE} {and Corollary \ref{cor:GJSNE}}), so {the new system} is in $WW$ .} \\ 
 \indent
 {(In more detail, as in Corollary \ref{cor:GJSNE}, let $x$ and $x'$ be paths that agree to level $k+1$ and pass through $v(k+1, 2k)$ and  $v(k+1, 2k+1)$, respectively, and 
 	such that the ordinal edge label for every edge in $x$ and in $x'$ after level $k+1$ is $2k$.
 	 In other words, $x$ and $x'$ enter the Morse component $MC(k)$ at its top and then follow its two sides all the way down.
 	  By Proposition 4.2, $x$ and $x'$ are depth $k$. 
 	  It is easily verified that for every $j>k$, there exists a time $m<0$ such that $T^mx$ and $T^mx'$ agree to level $j-1$ and are minimal from the root into $v(j, 2k)$ and $v(j, 2k+1)$, respectively.
 	  Therefore $x$ and $x'$ are depth $k$ and have long cuts.)}\\
\indent 
{We show now that for all even $k$ there are no longer any depth $k$ pairs, so the new system is (vacuously) in $DM2$ (and hence in $H2$).}
{The idea is that, as in $X$, the only candidates for depth $k$ pairs are paths down $MC(k)$, but the changed edge orders cause these pairs to have different $2$-codings, so they have become depth $1$.}\\
\indent
{All changes to $X$ were at levels $4$ and higher.  
{Moreover, for every $j \geq 2$}, the coding of $v(j, 3)$ by vertices at level $j-1$ remains the same in $Y$ as it was in $X$; and at vertices in $Y$ where this coding is different than it was in $X$, it is now the same as the coding at $v(j, 1)$.  
Hence, {if as in Observation \ref{obs:recog} we expand each symbol (vertex) $v(j,i)$ to $C_{j-1}(v(j, i))$, the effect on the {$\phi_j$-images} 
(strings on $D_j,E_j,D_{j-1},E_{j-1}$)
produces the same morphisms as before.}
The argument made in Example \ref{ex:modGJ} that paths in $X$ with the same $2$-coding are $2$-equivalent can then be applied to show the same holds true in $Y$. 
Hence, the proof of Proposition \ref{prop:GJSNE} {(4)} still works to show that for any $k \geq 2$ any depth $k$ pair in the new system represented by $Y$ is $k$-equivalent but not $(k+1)$-equivalent.}\\ 
\indent
{Now let $x$, $x'$ be a pair of paths in $Y$ that for some $k \geq 2$ is depth $k$ in the corresponding system.  
Since Observation \ref{obs:GJ} (2) still holds for $Y$, we can argue as in the proof of Proposition \ref{prop:GJSNE} {(3}) that $x$ and $x'$ in $Y$ have the same sequence of edge labels,  {enter $MC(k)$ at the level at which they first differ and are contained in $MC
(k)$ at all subsequent levels.}
{At any level of $Y$ only pairs of edges inside {one of the odd
Morse components in $Y$} {can} have the same label and distinct sources.}
Hence, {$k$ must be odd, and so} there are no pairs of paths with even depth in the new system.}\\
\indent
Telescoping the new diagram to a strictly increasing sequence of levels $j_n$ will produce a system that is not in $H2$ (or $DM2$).  
This is because given any $n$ there are an {odd} integer $k \in [j_n, j_{n+1})$ and a depth $k$ pair with long cuts. 
After the telescoping this pair will be depth $n$ (by {the proof of} Proposition \ref{prop:newlemma} (1)) and will still have long cuts (by Proposition \ref{prop:newlemma} (2)), so the telescoped system is in $W_0$. 
By definition, $W_0 \cap H2 = \emptyset$, so the class $H2$ is not closed under telescoping.\\ 
\indent
{Remark \ref{rem:classes} mentioned that $H2$ not being closed under telescoping impacts the proofs in \cites{dm2008,Hoynes2017}.
If in some system for every odd $k$ there were a depth $k$ pair with a cutoff (so the system is nonvacuously in $H2$, unlike our example), and for every even $k$ there were a depth $k$ pair with long cuts (so the system is also in $WW$), 
the same argument as above would show that telescoping to any strictly increasing sequence of levels produces a system that is not in $H2$.
If the levels were chosen to produce a system in $U_2$, the result would be in $U_2 \setminus H2$.
The proofs in \cites{dm2008,Hoynes2017} could be clarified by ruling out this case or dealing with it, maybe along the lines we suggest in Remark \ref{rem:classes}.}
\end{example}

 The following {Theorem} subsumes Proposition \ref{prop:noSNEconjOdom}.
 \begin{theorem}\label{thm:CWimpliesNCU}
  No system that is conjugate to a well timed system can also be conjugate to an untimed system: $CW \subset NCU$.
  \end{theorem}
  \begin{proof}
   Suppose that $X$ is a system {that is conjugate to a well timed system and is also conjugate to an untimed system $Y$.} 
     As mentioned above, by \cite[Theorem 4.7]{HPS1992},
  then there are a system $Z$ and telescopings $\tilde X$ of $X$ and $\tilde Y$ of $Y$ such that $Z$ telescopes on even levels $0,2,4,\dots$  to $\tilde X$ and on odd levels $0,1,3, \dots$ to $\tilde Y$. 
  
  By  Proposition \ref{prop:CW}, $X \in WW$. Repeated application of Lemma \ref{lem:lifts} gives $\tilde X \in WW$, then $Z \in WW$, then $\tilde Y \in WW$, and finally $Y \in WW$. 
  {This is a contradiction, since by definition no system can be both untimed and weakly well timed.}
    \end{proof}

\begin{theorem}\label{thm:disjunion}
	The family of nonexpansive systems is the disjoint union of those conjugate to well timed systems and those conjugate to untimed systems: 
	\be
	NE=CW \sqcup CU.
	\en
\end{theorem}
\begin{proof}
	By Theorem \ref{thm:CWimpliesNCU}, $CW \subset NCU$, so it
	 remains only to show that $NCW \subset CU$. 
	{By Proposition \ref{prop:CW}, if $(X,T) \in NCW$, then there is a $k_0$ such that for all $k \geq k_0$ there is a $k$ cutoff $j(k)>k$: no depth $k$ pair has a $j(k)$ cut. 
	 Let us telescope to levels $n_l, l \geq 0$, with $n_1 > k_0$ to produce the system $(\tilde X, \tilde T)$.
	 	Suppose that $\tilde k \geq 1$ and 
	 	 $\tilde x, \tilde y$ is a depth $\tilde k$ pair in $\tilde X$ with a $\tilde j$ cut.
	 	 Proposition \ref{prop:newlemma} (3) and (4) tell us that 
	 	 then $\tilde x, \tilde y$ is the image under the telescoping of a pair $x,y$ in $X$ of depth $i(\tilde x, \tilde y)
	 	 \in [n_{\tilde k}, n_{\tilde k +1})$ with a $J(\tilde j)=n_{\tilde j}$ cut. 
	 	 Every $i \in [n_{\tilde k}, n_{\tilde k +1})$ has a cutoff $j(i)$ in $X$, in particular 
	 	 $J(\tilde j) < j(i(\tilde x, \tilde y))$. 
	 	 	Since
	 	 	\be
	 	 	\tilde j \leq J(\tilde j) < j(i(\tilde x, \tilde y)) \leq \max \{j(i): i \in [n_{\tilde k}, n_{\tilde k +1})\},
	 	 	\en
	$\tilde j$ is bounded, over all pairs in $\tilde X$ of depth  $\tilde k$. Thus in $\tilde X$ for every $\tilde k$ there is a $\tilde k$ cutoff. This shows that	
	 telescoping past $k_0$ produces a system in $U$, and hence $(X,T) \in CU$.} 
\end{proof}
 
 \section{Describing untimed systems}\label{sec:untimed}

 % Below (\ref{prop:untimed1}--\ref{ex:nondet})
 Now we take a few steps towards determining exactly which $BV$ systems are very untimed.
 If a system $(X,T)$ is very untimed and not conjugate to an odometer, then, by \cite{dm2008}, it {must have} infinite ``topological rank". 
 {At the moment we do not have an example of a very untimed system that has unbounded width.}
 		 
  \begin{comment}
  \begin{example}\label{ex:untimedunbddwidth}
  Figure \ref{fig:untimedExample} shows an example of a very untimed system that is not
  conjugate to any odometer (and so has infinite ``topological rank"). 
  It is constructed by modifying the GJ example \cite[Figure 4]{GJ2000}.
  
  \begin{figure}
  \includegraphics[angle=90,origin=c,page=1,width=.65\textwidth]{DNCExample2020_05_15.pdf}
  \caption{An unbounded width very untimed system}
  \label{fig:untimedExample}
  \end{figure}
  \end{comment} 
  
  \begin{lemma}\label{lem:uo}
  If in a $BV$ diagram level {$n+1$} is uniformly ordered and has more than one vertex, 	
  then in the $BV$ system there is a depth $n$ pair of paths with an $n+1$ cut.
  \end{lemma}
  \begin{proof}
  Choose two paths $x,x'$ that are minimal into distinct vertices at level $n+1$.
  \begin{comment}  	
  and which are not in the orbit of the minimal/maximal path. 
  \end{comment}
  Since level $n+1$ is uniformly ordered, the $n$-basic block at each vertex at level $n+1$ is periodic with shortest repeating block $P$. 
  (By this we mean that for each vertex $v$ at level $n+1$, there is an integer $n_v$ such that the $n$-basic
  block at $v$ is $P^{n_v}$). 
  Thus the $n$-factor is a rotation on $|P|$ points, and the $n$-coding of every path is the two-sided sequence $P^\infty$, with a choice of the center position.
  Since $x$ and $x'$ are minimal into level $n+1$, their ``dot" is at the beginning of an explicit appearance of $P$ in the $n$-basic blocks at $v_{n+1}(x)$ and $v_{n+1}(x')$ respectively, 
  so $x,x'$ have the same $n$-coding. 
  Hence $x$ and $x'$ are depth $n$ with an $n+1$ cut.
  \end{proof}
  
  \begin{proposition}
  {Every bounded width very untimed system has only finitely many levels with more than one vertex.}
  \end{proposition}
  \begin{proof}
 {By \cite{dm2008}, bounded width and nonexpansive implies conjugate to an odometer. 
  By \cite{FPS2017} (and also mentioned in \cite{GJ2000}) there is 
  {a telescoping that has} infinitely many uniformly ordered levels. 
  Let $n+1$ ($n > 0$ ) be a uniformly ordered level in the telescoped system, which is still very untimed {(see Remark \ref{rem:untimedtel}, (2))}.
  Because of the Lemma and because no depth $n$ pair can have an $n+1$ cut, this level $n+1$ must have just one vertex. 
  Returning to the original (not telescoped) system, the level $m$ corresponding to level $n+1$ in the telescoped system has just one vertex.
  Any level that follows a level with just one vertex is uniformly ordered, and so, by the Lemma, level $m+1$ in the original system must also have just one vertex. 
  Thus every level in the original system from level $m+1$ on has just one vertex.}
  \end{proof}
  
  \begin{example}\label{ex:standardform}
  So one way to produce very untimed systems is to begin with levels with more than one vertex with edges ordered so that all the edges downward from each vertex have different ordinal labels. 
  Exiting these levels {downward} there can be no cuts. 
  Then continue forever with single-vertex levels connected by multiple edges. 
    Any such system is conjugate to an odometer, upon telescoping from the root to the first single-vertex level.\\
       \end{example}

  \begin{definition}\label{def:deterministic}
  A $BV$ system is {\em deterministic} if for every $n \geq 1$ the edges downward from each vertex at level $n$ have different ordinal labels.
  \end{definition} 
  
  The following Proposition provides more detail about the possible form of deterministic diagrams such as the one in Example \ref{ex:standardform}.

  \begin{proposition}
  Let $(X,T)$ be a deterministic system. 
  For each $n \geq 1$ let $k_n$ denote the number of vertices at level $n$. 
  Then $k_1 \geq k_2 \geq \dots$, and eventually all $k_n=1$. 
  (So the diagram has the form of the previous Example.)
  \end{proposition}
  \begin{proof}
  If $k_n < k_{n+1}$ for some $n \geq 1$, then the label $1$ must be repeated on edges between levels $n$ and $n+1$, so {$k_n \geq k_{n+1}$} for all $n$.
  
  If eventually all $k_n=k$, we claim that there would be $k$ minimal paths, and so we must have $k=1$.
  This is because from each low {(late)} enough level $n+1$ there are $k$ minimal edges (labeled $1$) {upward} to $k$ different vertices at level $n$ (because each vertex at level $n$ has a single edge downward that is labeled $1$), and hence $k$ minimal paths from level $n+1$ up to the root. 
  Looking at larger and larger values of $n$, we can see $k$ distinct paths that follow only minimal edges. 
  \end{proof}
  
  \begin{example}\label{ex:nondet}
  Not every bounded width very untimed system must be deterministic, as Figure \ref{fig:nondet} shows. 
  {The small numbers are ordinal edge labels, and the strings in  parentheses are $1$-basic blocks.}
  In this example the only possible cuts are {for
  pairs of paths among $x=0auA...$ (thick edges), $x'=0avB\dots$ (dashed edges) and $x''=0bwB\dots$ {(or pairs in the orbits of such pairs)}. 
  While $x$ and $x'$ have a $2$ cut, they are not depth $1$; $x$ and $x''$ have $2$ and $3$ cuts, but they are not depth $1$ or $2$; 
  and $x'$,$x''$ have a $2$ cut, but they are not depth $1$. In fact there are no depth $1$ or depth $2$ pairs.}
 
  \begin{figure}
  \begin{tikzpicture}[scale=1]
  \node at (0,0) (0) {$0$};
  \node at (-2,-1) (a) {$a$};
  \node at (0,-1) (b) {$b$};
  \node at (2,-1) (c) {$c$};
  \node at (-2,-2) (u) {$u\, (ab)$};
  \node at (-2.1,-1.6) {{\tiny $1$}};
  \node at (-1.6,-1.6) {{\tiny $2$}};
  \node at (-.55,-1.6) {{\tiny $1$}};
  \node at (-.15,-1.6) {{\tiny $2$}};
  \node at (0.9,-1.6) {{\tiny $1$}};
  \node at (1.8,-1.6) {{\tiny $2$}};
  \node at (-1.6,-2.6) {{\tiny $1$}};
  \node at (-.2,-2.6) {{\tiny $2$}};
  \node at (0.75,-2.6) {{\tiny $2$}};
  \node at (1.25,-2.6) {{\tiny $1$}};
  \node at (-.5,-3.7) {{\tiny $1$}};
  \node at (0.05,-3.7) {{\tiny $2$}};
  \node at (-2.1,-1.6) {{\tiny $1$}};
  \node at (-2.1,-1.6) {{\tiny $1$}};
  
  \node at (0,-2) (v) {$v\, (ab)$};
  \node at (2,-2) (w) {$w\, (bc)$};
  \node at (-1,-3) (A) {$A\, (abbc)$};
  \node at (1,-3) (B) {$B\,(bcab)$};
  \node at (0,-4) (AB) {};
  \draw  (AB) circle (0.05);
  \node at (.5,-4) (bk) {$AB$};
  \node at (0,-5) (x) {};
  \draw  (x)   (0.05);
  \node at (0,-6) (s) {};
  \draw  (s) circle (0.05);
  \node at (0,-6.2) (y) {};
  \draw  (y) circle (0.05);
  \node at (0,-6.4) (z) {};
  \draw  (z) circle (0.05);
  \node at (0,-6.6) (r) {};
  \draw  (r) circle (0.05);
  \draw[thick] (0) -- (a);
  \draw[dashed] (0) .. controls (-.5,-.5) .. (a);
  \draw (0) -- (b);
  \draw (0) -- (c);
  \draw[thick] (a) -- (u);
  \draw[dashed] (a) -- (v);
  \draw (b) -- (v);
  \draw (b) -- (u);
  \draw (b) -- (w);
  \draw (c) -- (w);
  \draw[thick] (u) -- (A);
  \draw[dashed] (v) -- (B);
  \draw (w) -- (A);
  \draw (w) -- (B); 
  \draw[thick] (A) -- (AB);
  \draw[dashed] (B) -- (AB);
  \draw[thick] (AB) .. controls (-.2,-4.5) .. (x);
  \draw[dashed] (AB) .. controls (.2,-4.5) .. (x); \draw[thick] (x) .. controls (-.2,-5.5) .. (s);
  \draw[dashed] (x) .. controls (.2,-5.5) .. (s);
  
  \end{tikzpicture}
  \caption{A nondeterministic bounded width very untimed system}
  \label{fig:nondet}
  \end{figure}
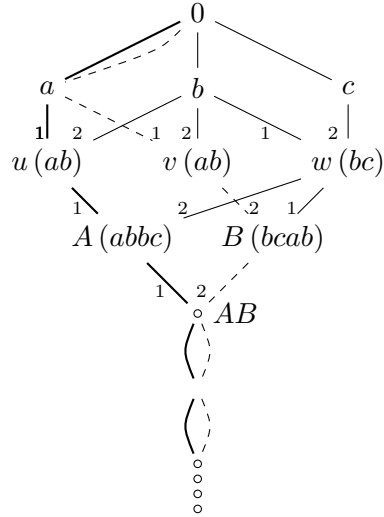
  \end{example}

\section{Conclusion and questions}
The family of simple, {properly} ordered Bratteli-Vershik systems is the disjoint union of the expansive systems, the systems conjugate to well timed systems, and the systems conjugate to untimed systems:
\be
BV = E \sqcup CW \sqcup CU .
\en

The foregoing suggests several questions:

1. Is there an example of an untimed (or even very untimed) system that has unbounded width? 
Is there an untimed system that has infinite topological rank, equivalently is not conjugate to any odometer? 
Is the class $U$ closed under conjugacy?

2. Is every well timed system conjugate to an $SSNE$ (or $SNE$) system?
If so, one could regularize nonexpansiveness by modifying any member of this class $W$ of well timed  $NE$ systems so that it becomes $SSNE$. 

3. There are many questions about the relations among the classes of systems that we have  defined, and others that could be defined. 
Exactly which classes can overlap, exactly what inclusions are there among them, exactly which of these inclusions are strict, etc.?
Is every odometer ({by which we here mean all levels strictly uniformly ordered, cf. Example \ref{ex:odoms})} in $U$? 
Is every system conjugate to an odometer in $U$? 
Is $SNE$ contained in $W$? 
Can $H2$ and $W$ intersect? (We know $DM2 \cap W = \emptyset$.)
Is $U_1 \cap W \neq \emptyset$? Is $U_2 \cap W \neq \emptyset$? 

\bibliographystyle{mscplain}
\bibliography{PeriodicBiblio2021_11_29}

 \end{document}